\documentclass[a4paper]{book}
\usepackage{amsmath,amsthm,amsfonts,amssymb}
\usepackage{framed}
\usepackage{xcolor}
\usepackage{graphicx}
\usepackage{subfigure}
\usepackage{algpseudocode}
\usepackage{units}
\usepackage{bbm}
\usepackage[bookmarks=true,hidelinks]{hyperref}

\newtheorem{theorem}{Theorem}[section]

\newtheorem{proposition}[theorem]{Proposition}

\theoremstyle{definition}
\newtheorem{algorithm}[theorem]{Algorithm}

\newtheorem{remark}[theorem]{Remark}
\allowdisplaybreaks

\numberwithin{equation}{section}

\newcommand{\eps} {\varepsilon}
\renewcommand{\epsilon}{\varepsilon}
\renewcommand{\phi}{\varphi}

\newcommand{\cell}{I}

\newcommand{\R}{\mathbb{R}}
\newcommand{\N}{\mathbb{N}}
\newcommand{\Z}{\mathbb{Z}}
\newcommand{\Dx}{{\Delta x}}

\newcommand{\Dt}{{\Delta t}}
\newcommand{\sgn}{{\rm sgn}}
\newcommand{\TV}{{\rm TV}}

\newcommand{\hf}{{\unitfrac{1}{2}}}
\newcommand{\thf}{{\unitfrac{3}{2}}}

\newcommand{\iphf}{{i+\hf}}
\newcommand{\imhf}{{i-\hf}}

\newcommand{\jphf}{{j+\hf}}
\newcommand{\jmhf}{{j-\hf}}

\newcommand{\jmpr}[1]{\langle\hspace*{-0.25em}\langle #1 \rangle\hspace*{-0.25em}\rangle}
\newcommand{\jmp}[1]{\lbrack\hspace*{-0.14em}\lbrack\hspace{0.05em} #1 \hspace{0.05em}\rbrack\hspace*{-0.14em}\rbrack}

\newcommand{\avg}[1]{\overline{#1}}

\newcommand{\avgv}{\avg{v}}
\newcommand{\mean}[1]{{\{#1\}}}
\newcommand{\recon}{\mathcal{R}}
\newcommand{\ind}{\mathbbm{1}} 

\renewcommand{\leq}{\leqslant}
\renewcommand{\geq}{\geqslant}

\def\Xint#1{\mathchoice
{\XXint\displaystyle\textstyle{#1}} 
{\XXint\textstyle\scriptstyle{#1}} 
{\XXint\scriptstyle\scriptscriptstyle{#1}} 
{\XXint\scriptscriptstyle\scriptscriptstyle{#1}} 
\!\int}
\def\XXint#1#2#3{{\setbox0=\hbox{$#1{#2#3}{\int}$ }
\vcenter{\hbox{$#2#3$ }}\kern-.57\wd0}}
\def\intavg{\Xint-}

\begin{document}

\chapter*{Stability properties of the ENO method}
\textbf{\large
Ulrik Skre Fjordholm\footnote{Department of Mathematical Sciences, NTNU, 7491 Trondheim, Norway. \newline Email: ulrik.fjordholm@math.ntnu.no}$^,$\footnote{Research supported in part by the grant \textit{Waves and Nonlinear Phenomena} (WaNP) from the Research Council of Norway.}
}

{\small
\textbf{Abstract} 
We review the currently available stability properties of the ENO reconstruction procedure, such as its monotonicity and non-oscillatory properties, the sign property, upper bounds on cell interface jumps and a total variation-type bound. We also outline how these properties can be applied to derive stability and convergence of high-order accurate schemes for conservation laws.
}

\section{Introduction}
The ENO (Essentially Non-Oscillatory) reconstruction method is a method of recovering---to a high degree of accuracy---a function $v$, given only discrete information such as a finite number of point values $v_i = v(x_i)$ or local averages
\[
\avgv_i = \intavg_{\cell_i} v(x)\ dx, \qquad i\in\Z
\]
over \textit{cells} $I_i = [x_\imhf, x_\iphf)$. (Here and below we will denote $\intavg_\cell = \frac{1}{|\cell|}\int$.) The method was first developed as a means of increasing the order of accuracy of numerical methods for hyperbolic conservation laws. Solutions of these types of PDEs are at best \textit{piecewise} smooth, and can have large jump discontinuities. The ENO method accomplishes the feat of approximating $v$ to a high degree of accuracy in smooth parts, while avoiding ``Gibbs-like'' oscillations near the discontinuities. The purpose of this paper is to review the currently known stability properties of the ENO method, and the application of these to numerical methods for hyperbolic conservation laws.

The ENO reconstruction method was originally developed by Ami Harten~\cite{Har86} and further developed and analyzed by Harten, Osher, Engquist and Chakravarthy in a series of papers \cite{HO87,HEOC86,HEOC87}. Since then it has been generalized and applied to a number of areas. In this paper we will concentrate on the one-dimensional version of the ENO reconstruction method, and its application to approximate one-dimensional scalar conservation laws. Thus, we leave out a large body of work on multi-dimensional generalizations of ENO, related ``ENO-type'' reconstruction methods, and applications of ENO to systems of conservation laws, as well as other fields such as data compression/representation and image analysis/reconstruction. Multi-dimensional ENO methods were introduced by Shu and Osher on Cartesian (tensor-product) meshes \cite{SO88}, and generalized to unstructured (triangular) meshes by Harten and Chakravarthy \cite{HC91} and by Abgrall and Lafon \cite{AL93} (see also \cite{Abg94}). Related ``ENO-type'' methods include the highly successful Weighted ENO method \cite{LOC94,JS96}, biased ENO \cite{Shu90}, ENO-SR (subcell resolution) \cite{Har89} and its multi-dimensional generalization GENO (geometric ENO) \cite{SKS97}, and ENO-EA (edge adapted) \cite{ACDDM08}. For applications of ENO apart from conservation laws we mention in particular Harten's work on multiresolution methods \cite{Har96}; see also \cite{AD00}.

%

Here follows an outline of the rest of the paper. In Section \ref{sec:enorecon} we motivate and describe the ENO reconstruction method. In Section \ref{sec:conslaws} we briefly describe the application of the ENO method to (scalar) conservation laws; we show that the resulting second-order accurate scheme is convergent; and we derive a list of \textit{a priori} bounds that imply convergence of (one class of) higher-order ENO-type schemes. Section \ref{sec:enostab} is the main section of the paper. We start by listing some immediate stability properties of the ENO method, and move on to describing some of the more non-trivial properties such as the sign property, upper bounds on jumps and the ``essentially non-oscillatory'' property.

We have attempted to make this paper as self-contained as possible. In particular, Sections \ref{sec:enorecon} and \ref{sec:enostab} should be accessible also to readers without a background in PDEs.

\section{The ENO reconstruction method}\label{sec:enorecon}
For the sake of completeness we describe here the ENO reconstruction method. We refer to the review article by Shu and Zhang \cite{SZ16} for further details.

Let us fix a partition $(\cell_i)_{i\in\Z}$ of the real line, where each \textit{cell} $\cell_i$ is an interval $\cell_i = [x_\imhf, x_\iphf)$ of length $\Dx_i = x_\iphf-x_\imhf$, bounded from above by $\Dx=\max_i\Dx_i$. Let $(\avgv_i)_{i\in\Z}\subset\R$ be a given collection of numbers, which we interpret as the cell averages of some unknown function $v$,
\begin{equation}\label{eq:avgvdef}
\avgv_i = \intavg_{\cell_i} v(x)\ dx.
\end{equation}
The ENO (Essentially Non-Oscillatory) reconstruction method \cite{Har86,HEOC87} aims to \emph{reconstruct} $v$ by producing a collection of $(k-1)$th order polynomials $p_i = p_i(x)$ which approximate $v$ to $k$th order:
\begin{equation}\label{eq:kthorderapprox}
p_i(x) = v(x) + e(x)\Dx_i^{k} \qquad \forall\ x\in\cell_i,
\end{equation}
where $e(x)$ denotes the leading-order error term. The reconstruction is required to conserve mass, in the sense that
\begin{equation}
\intavg_{\cell_i} p_i(x)\ dx = \avgv_i \qquad \forall\ i\in\Z,
\end{equation}
and is required to be as ``non-oscillatory'' as possible.

The properties of accuracy and mass conservation are automatically satisfied if $p_i$ interpolates the cell average values $\avgv_j$ over any of the $k$ \emph{stencils}
\[
\{s,\dots,s+k-1\}, \qquad i-k+1 \leq s \leq i.
\]
Thus, $p_i$ is defined as the unique $(k-1)$th order polynomial which satisfies
\begin{equation}\label{eq:cellavginterp}
\intavg_{\cell_j} p_i(x)\ dx = \avgv_j \qquad \text{for } j=s,\dots,s+k-1,
\end{equation}
for some integer $s = s_i\in\{i-k+1,\dots,i\}$ called the \emph{stencil index}. The problem of finding $p_i$ satisfying \eqref{eq:cellavginterp} is a somewhat nonstandard interpolation problem, and Harten \cite{Har86} suggested two approaches. 

In the \textit{reconstruction via deconvolution} (RD) approach, it is observed that \eqref{eq:avgvdef} is a convolution of $v$ with the indicator function over $\cell_i$. Taylor expanding $v$ and comparing with \eqref{eq:cellavginterp} results in an upper triangular linear system for $p_i$. 

In the \textit{reconstruction via primitive function} (RP) approach we define the primitive of $v$,
\begin{equation}\label{eq:primitive}
V(x) = \int_{-\infty}^x v(x)\ dx
\end{equation}
(the lower limit of this integral is irrelevant), and observe that $V$ is precisely known at every cell interface,
\[
V(x_\iphf) = \sum_{j\leq i} \Dx_j\avgv_j.
\]
If we let $P_i$ be the unique $k$th order polynomial which interpolates $V$ over the points $\{x_{s-\hf}, \dots, x_{s+k-\hf}\}$, then the $(k-1)$th order polynomial $p_i(x) = \frac{d}{dx}P_i(x)$ satisfies \eqref{eq:cellavginterp}.

The RD approach requires a uniform mesh (i.e.\ $\Dx\equiv$ const.), whereas the RP approach works for any (one-dimensional) mesh. Even on a uniform mesh, the RD and RP approaches are \textit{not} equivalent, i.e.\ they yield distinct reconstructions $p_i$. We are unaware of any further work on the RD methodology beyond the original papers by Harten et al.\ \cite{HO87,HEOC86,HEOC87}, and we will concentrate on RP for the remainder of this paper. (See also Remark \ref{rem:rdsignprop}.)

\subsubsection*{Choosing the stencil index}
The algorithm to select the stencil index $s_i$ is what characterizes the ENO reconstruction procedure. A naive choice of the stencil index could be the all-upwind or all-downwind stencils $s_i\equiv i-k+1$ or $s_i\equiv i$; however, given the possible non-smoothness or discontinuity of $v$, these choices would lead to ``Gibbs phenomena''---large oscillations in non-smooth regions. 

Harten \cite{Har86} proposed an iterative, data-dependent algorithm to compute $s_i$. The algorithm is based upon the divided differences of $V$, defined as
\[
\begin{cases}
V[x_\iphf] = V(x_\iphf) \\
V[x_\imhf, \dots, x_\jphf] = \frac{V[x_\iphf, \dots, x_\jphf] - V[x_\imhf, \dots, x_\jmhf]}{x_\jphf - x_\imhf} & \forall\ i<j.
\end{cases}
\]
Starting with the stencil $\{x_\imhf, x_\iphf\}$, the ENO stencil selection procedure adds either the left or right point $x_{i-\thf}$ or $x_{i+\thf}$, depending on which of the divided differences $V[x_{i-\thf},x_{i-\hf},x_\iphf]$ or $V[x_\imhf,x_\iphf,x_{i+\thf}]$ is the smallest. This process is then iterated, ending up with an interpolation stencil $\{x_{s_i-\hf}, \dots, x_{s_i+k+\hf}\}$ for some $s_i\in\{i-k+1,\dots,i\}$. Recalling that the $k$th divided difference of $V$ is an approximation of the $(k-1)$th derivative of $v$,
\[
V[x_{i-\hf}, \dots, x_{i+k-\hf}] = \frac{V^{(k)}(\xi)}{k!} = \frac{v^{(k-1)}(\xi)}{k!}, \qquad \xi\in[x_\imhf, x_{i+k-\hf}],
\]
we see that the ENO procedure iteratively adds a new point to the interpolation stencil ``in the direction of smoothness''.


Since $V[x_\imhf,x_\iphf] = \avg{v}_i$, we can write
\[
V[x_\imhf, \dots, x_\jphf] = [\avgv_i,\dots,\avgv_j] \qquad \forall\ i\leq j
\]
where the ``cell-average divided differences'' are defined as
\begin{equation}\label{eq:avgdivdiff}
\begin{cases}
[\avgv_{i}] = \avgv_i \\
[\avgv_{i},\dots,\avgv_{j}] = \frac{[\avgv_{i+1},\dots,\avgv_{j}] - [\avgv_{i},\dots,\avgv_{j-1}]}{x_{j+\hf}-x_{i-\hf}} & \forall\ i<j.
\end{cases}
\end{equation}
We summarize the ENO procedure using this notation.

\begin{framed}
\begin{algorithm}\label{alg:eno}
\textbf{(ENO Stencil Selection Procedure)}
\begin{algorithmic}
\State{$s_i^1 = 0$}
\For{$\ell=1,\dots,k-1$}
	\If{$\left|\bigl[\avgv_{s_i^\ell-1}, \dots, \avgv_{s_i^\ell+\ell-1}\bigr]\right| < \left|\bigl[\avgv_{s_i^\ell}, \dots, \avgv_{s_i^\ell+\ell}\bigr]\right|$}
		\State{$s_i^{\ell+1} = s_i^\ell-1$}
	\Else
		\State{$s_i^{\ell+1} = s_i^\ell$}
	\EndIf
\EndFor
\State{$s_i = s_i^k$}
\State {Let $P_i$ interpolate $V$ over $\{x_{s_i-\hf},\dots,x_{s_i+k-\hf}\}$}
\State {Define $p_i(x) = \frac{d}{dx}P_i(x)$}
\end{algorithmic}
\end{algorithm}
\end{framed}

The implications of the ENO stencil selection procedure are easiest to see with the Newton form of the interpolating polynomial $P_i$. It is straightforward to show by induction that the Newton form of $P_i$ can be expressed as
\[
P_i(x) = \sum_{\ell=0}^k V[x_{s_i^\ell-\hf}, \dots, x_{s_i^\ell+\ell-\hf}] \prod_{m=0}^{\ell-1}(x-x_{s_i^{\ell-1}+m-\hf}),
\]
where we have defined $s_i^{-1}=s_i^0 = i$. After differentiating and using the notation \eqref{eq:avgdivdiff} we get
\begin{equation}\label{eq:enoexpression}
\begin{split}
p_i(x) &= \sum_{\ell=1}^k V[x_{s_i^\ell-\hf}, \dots, x_{s_i^\ell+\ell-\hf}] \sum_{n=0}^{\ell-1}\prod_{\substack{m=0\\m\neq n}}^{\ell-1}(x-x_{s_i^{\ell-1}+m-\hf}) \\*
&= \sum_{\ell=1}^{k} [\avgv_{s_i^{\ell}}, \dots, \avgv_{s_i^{\ell}+\ell-1}] \sum_{n=0}^{\ell-1}\prod_{\substack{m=0\\m\neq n}}^{\ell-1}(x-x_{s_i^{\ell-1}+m-\hf})
\end{split}
\end{equation}
(see also \cite[p.\ 81]{Har86}). Thus, the ENO stencil selection procedure chooses each index $s_i^\ell$ so that the above coefficients $[\avgv_{s_i^{\ell}}, \dots, \avgv_{s_i^{\ell}+\ell-1}]$ are as small as possible, thereby obtaining the least oscillatory polynomial possible.

Note that both the ENO stencil selection procedure and the formula for $p_i$ can be written completely in terms of the divided differences of $\avgv$. Thus, it is not necessary to compute the primitive $V$ or its divided differences.

\begin{remark}\label{rem:pointeno}
There is also a point-value version of the ENO reconstruction method. Given the point-values $v_i=v(x_i)$ of some function $v$, this method employs a similar algorithm to obtain a reconstruction $p_i(x) = v(x) + O(\Dx_i^k)$. The reconstruction $p_i$ is given by the $(k-1)$th order polynomial interpolating $(v_j)_{j\in\Z}$ over the points $x_{s_i},\dots,x_{s_i+k-1}$, where $s_i$ is obtained by replacing every occurence of $\avgv_j$ in Algorithm \ref{alg:eno} by $v_j$. See \cite{SO88} for further details and \cite{Fjo13,FMT12a} for a stability analysis.
\end{remark}

\section{Application to conservation laws}\label{sec:conslaws}
The ENO method was originally developed as a means of increasing the order of accuracy of finite volume schemes for hyperbolic conservation laws. We consider here only one-dimensional, scalar conservation laws
\begin{equation}\label{eq:cl}
\begin{split}
\partial_t u + \partial_x f(u) = 0 \\
u(x,0) = u_0(x).
\end{split}
\end{equation}
To establish the notation and some useful identities, we briefly review this setting in Section \ref{sec:fvm}. We refer to the article by Shu and Zhang \cite{SZ16} for further details. In Section \ref{sec:tvdeno} we see that the second-order ENO method results in a TVD, convergent finite volume scheme for scalar conservation laws. In Section \ref{sec:convhighorder} we review one approach to obtaining convergent higher-order accurate schemes. 

Below we use the notation from Section \ref{sec:enorecon}. Furthermore, we denote
\begin{equation}
\jmp{\avgv}_\iphf = \avgv_{i+1}-\avgv_i, \qquad \mean{\avgv}_\iphf = \frac{\avgv_i+\avgv_{i+1}}{2}.
\end{equation}

\subsection{Finite volume methods}\label{sec:fvm}
A (semi-discrete) finite volume method for \eqref{eq:cl} aims to compute an approximation of the cell averages
\[
\avgv_i(t) \approx \intavg_{\cell_i} u(x,t)\,dx \qquad \forall\ t\geq0
\]
of the exact (entropy) solution of \eqref{eq:cl}. A consistent, conservative finite volume method for \eqref{eq:cl} is then of the form
\begin{equation}\label{eq:fvmsd}		
\frac{d}{dt}\avgv_i(t) = - \frac{1}{\Dx_i}\bigl(F_\iphf - F_\imhf\bigr)
\end{equation}
for some $F_\iphf = F\bigl(\avgv_{i-m+1},\dots,\avgv_{i+m}\bigr)$, and $F$ is a numerical flux function such as the Godunov, Lax--Friedrichs or Engquist--Osher fluxes. One class of (formally) high-order accurate schemes is obtained by letting 
\begin{equation}\label{eq:musclflux}
F_\iphf = F\bigl(v_\iphf^-,\, v_\iphf^+\bigr)
\end{equation}
for some monotone flux $F$. Here, $v_\iphf^\pm$ are the reconstructed cell interface values
\begin{equation}\label{eq:reconcellinterface}
v_\iphf^- = p_i(x_\iphf,t), \quad v_\iphf^+ = p_{i+1}(x_\iphf,t), \quad p(x,t) = \recon(\avgv(t))(x)
\end{equation}
for some reconstruction operator $\recon$ such as ENO. 

To obtain a fully discrete method, we discretize the temporal domain $t\in[0,\infty)$ into discrete points $t^n = n\Dt$ for some $\Dt>0$ (which we for simplicity assume is constant), and the aim is to approximate
\[
\avgv_i^n \approx \intavg_{\cell_i} u(x,t^n)\,dx \qquad \forall\ i\in\Z.
\]
An explicit, fully discrete finite volume method for \eqref{eq:cl} is then of the form
\begin{equation}\label{eq:fvm}		
\avgv_i^{n+1} = \avgv_i^n - \frac{\Dt}{\Dx_i}\bigl(F_\iphf^n - F_\imhf^n\bigr)
\end{equation}
for some $F_\iphf^n = F\bigl(\avgv_{i-m+1}^n,\dots,\avgv_{i+m}^n\bigr)$. This scheme is \emph{total variation diminishing} (TVD) if
\begin{equation}\label{eq:tvd}
\TV(\avgv^{n+1}) \leq \TV(\avgv^n),
\end{equation}
so-called after Harten \cite{Har83}.

The scheme \eqref{eq:fvm} can be viewed as a (first-order accurate) Forward Euler discretization of \eqref{eq:fvmsd} (see \cite[Section II.3.3]{GR91} for a rigorous derivation; cf.\ also \cite[p.\ 352]{HEOC86}). Higher-order accurate methods can be obtained using multi-step methods or Strong Stability Preserving (SSP) Runge--Kutta methods \cite{GST01}, which consist of convex combinations of \eqref{eq:fvm}.

\subsection{TVD ENO schemes}\label{sec:tvdeno}
Consider now the (formally) second-order accurate scheme \eqref{eq:fvm} with a flux \eqref{eq:musclflux} using a second-order reconstruction method. Any second-order reconstruction $(p_i)_{i\in\Z}$ of cell averages $(\avgv_i)_{i\in\Z}$ must necessarily be of the form
\begin{equation}\label{eq:lininterp}
p_i(x) = \avgv_i + \sigma_i(x-x_i)
\end{equation}
where $\sigma_i\in\R$ is the slope of $p_i$. This slope is commonly written in the \emph{slope limited} form
\begin{equation}\label{eq:slopedef}
\sigma_i = \phi(\theta_i^+)\jmp{\avgv}_\iphf, \qquad \theta_i^+ = \frac{\jmp{\avgv}_\imhf}{\jmp{\avgv}_\iphf}
\end{equation}
for some $\phi : \R\to\R$ called a \emph{slope limiter}. 
Using Harten's work \cite{Har83}, Sweby \cite{Swe84} showed that if the slope limiter satisfies
\begin{equation}\label{eq:swebycondition}
\left|\phi(\theta_1) - \frac{\phi(\theta_2)}{\theta_2}\right| \leq 2 \qquad \forall\ \theta_1,\theta_2\in\R
\end{equation}
then the explicit discretization \eqref{eq:fvm} is both total variation diminishing (TVD) and uniformly bounded, so the computed solution satisfies
\[
\TV(\avgv^n) \leq \TV(\avgv^0), \qquad \|\avgv^n\|_{L^\infty} \leq \|\avgv^0\|_{L^\infty} \qquad \forall\ n\in\N.
\]
As a consequence, there is a subsequence $\Dt_m, \Dx_m \to 0$ for which the computed solutions converge towards a weak solution.

It is not hard to see that the second-order ENO reconstruction can be written as \eqref{eq:lininterp}, \eqref{eq:slopedef} with the slope limiter
\begin{equation}\label{eq:enolimiter}
\phi(\theta) = \begin{cases}
\theta & \text{if } |\theta|<1 \\
1 & \text{if } |\theta|\geq1.
\end{cases}
\end{equation}
Although this limiter does not lie in the ``TVD region'' introduced by Sweby \cite{Swe84}, it \textit{does} satisfy \eqref{eq:swebycondition}. Therefore, the scheme \eqref{eq:fvm}, \eqref{eq:musclflux} using second-order ENO reconstruction is both TVD and uniformly bounded, and hence converges (subsequentially) towards a weak solution.

\subsection{Convergence of high-order schemes}\label{sec:convhighorder}
A uniform bound on the total variation of a sequence of approximate solutions---such as the bound \eqref{eq:tvd} provided by TVD schemes---prevents the buildup of high-frequency oscillations, a necessary requirement for the strong convergence of the sequence. However, it is well-known that any TVD scheme for \eqref{eq:cl} is at most second-order accurate when measured in $L^1$. Thus, any proof of stability or convergence of higher (than second) order accurate schemes must necessarily relax the TVD requirement, while still preventing high-frequency oscillations.

We present here one class of convergent, high-order accurate schemes, the so-called TECNO schemes \cite{Fjo13,FMT12b}. As a motivation we first derive the necessary \textit{a priori} bounds for a parabolic regularization of \eqref{eq:cl}, which can be thought of as the effective (modified) equation of the numerical scheme. We then perform the analogous computations for the TECNO schemes.

\subsubsection{Motivation}
Consider the following regularization of \eqref{eq:cl}:
\begin{equation}\label{eq:clreg}
\begin{split}
\partial_t v^\eps &+ \partial_x f(v^\eps) = \eps \partial_{xx}v^\eps \\
& v^\eps(x,0) = v_0^\eps(x)
\end{split}
\end{equation}
(where $v_0^\eps$ converges to $u_0$ as $\eps\to0$). The term $\eps\partial_{xx}v^\eps$ can be thought of as the numerical viscosity of a numerical scheme, and $\eps \sim \Dx^k$, where $k$ is the order of accuracy of the method. Multiplying \eqref{eq:clreg} by $2v^\eps$ we obtain
\begin{equation}\label{eq:energyregular}
\partial_t (v^\eps)^2 + \partial_x q(v^\eps) = \eps\partial_{xx}(v^\eps)^2 - 2\eps(\partial_xv^\eps)^2,
\end{equation}
where $q$ satisfies $q'(u)=2uf'(u)$ for all $u\in\R$. Integrating \eqref{eq:energyregular} over $x\in\R$, $t\in[0,T]$ gives
\begin{equation}\label{eq:energyregintegrated}
\int_\R v^\eps(x,T)^2\ dx = \int_\R u_0^\eps(x)^2\ dx - 2\eps \int_0^T\int_\R(\partial_x v^\eps)^2\ dxdt.
\end{equation}
Thus, we have the two bounds
\begin{subequations}\label{eq:aprioribounds}
\begin{equation}
\|v^\eps(T)\|_{L^2(\R)} \leq \|u_0\|_{L^2(\R)}
\end{equation}
\begin{equation}\label{eq:regwtvbound}
2\eps\int_0^T\int_\R(\partial_x v^\eps)^2\ dxdt \leq \|u_0\|_{L^2(\R)}^2
\end{equation}
\end{subequations}
for all $\eps>0$, i.e., a uniform $L^2$ bound and a ``weak TV bound''. From these, compensated compactness techniques can be used to show that a subsequence $v^{\eps'}$ converges to a weak solution of \eqref{eq:cl} as $\eps'\to0$. Since the second term on the right-hand side of \eqref{eq:energyregular} is non-positive, we find that any strong limit $u=\lim_{\eps'\to0} v^{\eps'}$ satisfies the entropy condition
\begin{equation}\label{eq:entrcond}
\partial_t u^2 + \partial_x q(u) \leq 0.
\end{equation}
We conclude that the \textit{whole} sequence $(v^\eps)_{\eps>0}$ converges strongly to the (unique) entropy solution of \eqref{eq:cl}.

\subsubsection{TECNO schemes}\label{sec:tecno}
We consider now the semi-discrete finite volume method \eqref{eq:fvmsd} with a numerical flux function of the form
\begin{equation}\label{eq:entrstabflux}
F_\iphf = F^*_\iphf - c_\iphf\jmpr{v}_\iphf.
\end{equation}
Here, $\jmpr{v}_\iphf = v_\iphf^+ - v_\iphf^-$ is the cell interface jump in the reconstructed values (cf.\ \eqref{eq:reconcellinterface}) for some reconstruction operator $\recon$, to be determined. The diffusion constant $c_\iphf$ is some number satisfying $c_{\max}\geq c_\iphf\geq c_{\min}>0$, and $F^*$ is a Lipschitz continuous numerical flux, to be determined. Note that if the reconstructed values satisfy, say,
\begin{equation}\label{eq:upperjmpbound}
|\jmpr{v}_\iphf| \leq C|\jmp{\avgv}_\iphf|
\end{equation}
for some $C>0$ independent of $v$, then $F$ is Lipschitz continuous---a natural assumption in the convergence analysis of finite volume schemes. 

Multiplying \eqref{eq:fvmsd} by $2\avgv_i(t)$ we obtain
\begin{align*}
\frac{d}{dt}\avgv_i^2 + 2\avgv_i\frac{F^*_\iphf - F^*_\imhf}{\Dx_i} 
&= 2\avgv_i\frac{c_\iphf\jmpr{v}_\iphf - c_\imhf\jmpr{v}_\imhf}{\Dx_i} \\
&= 2\frac{c_\iphf\mean{v}_\iphf\jmpr{v}_\iphf - c_\imhf\mean{v}_\imhf\jmpr{v}_\imhf}{\Dx_i} \\
&\quad - \frac{c_\iphf\jmp{\avgv}_\iphf\jmpr{v}_\iphf + c_\imhf\jmp{\avgv}_\imhf\jmpr{v}_\imhf}{\Dx_i}.
\end{align*}
\textit{Assuming} that we can write $2\avgv_i(F^*_\iphf - F^*_\imhf) = (Q^*_\iphf - Q^*_\imhf)$ (as in the step from \eqref{eq:clreg} to \eqref{eq:energyregular}) for some ``numerical entropy flux $Q^*$'', we can define $Q_\iphf = Q^*_\iphf - 2c_\iphf\mean{\avgv}_\iphf\jmpr{v}_\iphf$ and obtain
\begin{equation}\label{eq:tecnoenergyestimate}
\frac{d}{dt}\avgv_i^2 + \frac{Q_\iphf - Q_\imhf}{\Dx_i} =  -\frac{c_\iphf\jmp{\avgv}_\iphf\jmpr{v}_\iphf + c_\imhf\jmp{\avgv}_\imhf\jmpr{v}_\imhf}{\Dx_i}.
\end{equation}
Summing over $i\in\Z$ and integrating over $t\in[0,T]$, we get
\begin{equation}
\sum_{i\in\Z} \avgv_i(T)^2\Dx_i = \sum_{i\in\Z}\avgv_i(0)^2\Dx_i - 2\int_0^T\sum_{i\in\Z}c_\iphf\jmp{\avgv}_\iphf\jmpr{v}_\iphf\, dt
\end{equation}
(compare with \eqref{eq:energyregintegrated}). Assuming now that 
\begin{equation}\label{eq:signprop}
\jmp{\avgv}_\iphf\jmpr{v}_\iphf \geq 0 \qquad \forall\ i\in\Z,
\end{equation}
i.e.\ that the jumps $\avgv_{i+1}-\avgv_i$ and $v_\iphf^+ - v_\iphf^-$ have the same sign, we can conclude that
\begin{subequations}
\begin{equation}
\|v_\Dx(T)\|_{L^2(\R)} \leq \|v_\Dx(0)\|_{L^2(\R)},
\end{equation}
\begin{equation}\label{eq:wrecontvbound}
2\int_0^T\sum_{i\in\Z}c_\iphf\jmp{\avgv}_\iphf\jmpr{v}_\iphf\,dt \leq \|v_\Dx(0)\|_{L^2(\R)}^2
\end{equation}
\end{subequations}
(compare with \eqref{eq:aprioribounds}). The property \eqref{eq:signprop} also ensures that the right-hand side of \eqref{eq:tecnoenergyestimate} is non-positive, so that
\[
\frac{d}{dt}\avgv_i^2 + \frac{Q_\iphf - Q_\imhf}{\Dx_i} \leq 0
\]
(compare with \eqref{eq:entrcond}), i.e.\ a discrete entropy inequality is satisfied.

The bound \eqref{eq:wrecontvbound} is not quite a weak TV bound like \eqref{eq:regwtvbound}---for this we would need a bound of the form
\begin{equation}\label{eq:wtvbound}
\int_0^T\sum_{i\in\Z}|\jmp{\avgv}_\iphf|^p\,dt \leq C
\end{equation}
for some $p\geq1$ and $C>0$ independent of $\Dx$. 

We have thus arrived at a list of properties which enable a convergence proof of the finite volume method \eqref{eq:fvm}: The upper bound on reconstructed jumps \eqref{eq:upperjmpbound}, the \textit{sign property} \eqref{eq:signprop}, and the ``weak TV bound'' \eqref{eq:wtvbound}. 

The sign property and the upper bound have been proven for the ENO reconstruction method and are discussed in Sections \ref{sec:signprop} and \ref{sec:upperjmpbound}, respectively. For $k=2$ it has been proven---and conjectured for $k>2$---that the ``reconstructed TV bound'' \eqref{eq:wrecontvbound} implies the ``weak TV bound'' \eqref{eq:wtvbound}. This is discussed in Section \ref{sec:enoconjecture}. We refer to this conjecture as the \textit{ENO TV conjecture}.

In \cite{Fjo13,FMT12b} the authors constructed schemes of the form \eqref{eq:fvmsd}, \eqref{eq:entrstabflux} which uses the ENO reconstruction method---the so-called \textit{TECNO schemes}. We summarize the main convergence theorem here and refer to \cite{Fjo13} for the proof.

\begin{theorem}
For every $k$ for which the ENO TV conjecture holds, we have the following. If the approximate solution computed by the $k$th order TECNO method is $L^\infty$-bounded, then the sequence of approximate solutions converges to the entropy solution of \eqref{eq:cl} as $\Dx\to0$.
\end{theorem}

\begin{remark}
With some extra effort, the above computation can be generalized from the square entropy $v^2$ to arbitrary entropies $\eta(v)$. See the review article by Tadmor \cite{Tad16} (see also \cite{Tad03,FMT12b}) for more information on so-called entropy stable methods.
\end{remark}

\section{ENO stability properties}\label{sec:enostab}
In this section we review the currently known stability properties of the ENO reconstruction method. In Section \ref{sec:easyprops} we summarize some immediate (but nevertheless useful) properties of the ENO reconstruction. In Section \ref{sec:signprop} we prove the \textit{sign property} of the ENO method, and in Section \ref{sec:upperjmpbound} we prove an upper bound on the jump $\jmpr{v}=v_\iphf^+-v_\iphf^-$. We discuss the ENO TV conjecture in Section \ref{sec:enoconjecture}. Recall from Section \ref{sec:convhighorder} that all of these properties are essential for the convergence of the high-order TECNO schemes.

In Section \ref{sec:enomeshdep} we prove some well-known \textit{mesh dependent} properties of ENO. As it turns out, the sign property is a necessary ingredient in a rigorous proof of some of these properties. We conclude in Section \ref{sec:enodeficiency} by mentioning some deficiencies of ENO.

\subsection{Immediate properties}\label{sec:easyprops}

\subsubsection{Mesh invariance and linearity}\label{sec:enolinear}
Under the mapping $x \to a+bx$ for any $a\in\R$ and $b>0$, the reconstructed polynomial is $p_i(\frac{x-a}{b})$.
If $(\avgv_i)_{i\in\Z}$ is replaced by $(\alpha\avgv_i + \beta)_{i\in\Z}$ for any $\alpha,\beta\in\R$, then the ENO reconstruction $p_i(x)$ is replaced by $\alpha p_i(x)+\beta$.

\subsubsection{Discontinuity across cell edges}\label{sec:disccelledges}
As a rule of thumb, the ENO reconstruction $p=\recon(\avgv)$ is discontinuous \textit{at least} at every $k$th cell interface $x_\iphf$. To see this, note that neighboring cells with the same stencil index $s_i=s_{i+1}$ have the same reconstruction $p_i=p_{i+1}$ (and are thus continuous at $x_\iphf$), whereas if $s_i<s_{i+1}$ then $p_i\neq p_{i+1}$, and hence $p_i(x_\iphf)\neq p_{i+1}(x_\iphf)$ (except in very rare cases, such as when $v$ is itself a $(k-1)$th order polynomial). Since $s_i$ must change at least at every $k$th index $i$, this yields a discontinuity in $p$.

At points of discontinuity $x_\iphf$, the size of the jump $p_{i+1}(x_\iphf)-p_i(x_\iphf)$ is $O(\Dx^k)$ (see Section \ref{sec:signprop}). Note that the cell interface jump $p_{i+1}(x_\iphf)-p_i(x_\iphf)$ can---and often will---be zero even when $\avgv_{i+1}-\avgv_i\neq 0$.

\subsubsection{Uniform $k$th order accuracy}\label{sec:uniformacc}
Let $v\in C^\infty(\R)$ with primitive $V(x)$ defined in \eqref{eq:primitive}. Through a Taylor expansion of $V$ it is easy to see that the ENO reconstruction $p=\recon(\avgv)$ of $(\avgv_i)_{i\in\Z}$ is a $k$th order approximation of $v$. More specifically, $p_i$ satisfies the relation \eqref{eq:kthorderapprox} with an error term $|e(x)| \leq C\|\frac{d^k v}{dx^k}\|_{L^\infty}$ for some $C = C_k$. In each cell $\cell_i$, the error term $e(x)$ is continuous (but not Lipschitz continuous) with at least one zero. It is discontinuous only at those cell interfaces $x_\iphf$ where $p$ is discontinuous (see Section \ref{sec:disccelledges}).

\subsection{The sign property}\label{sec:signprop}
Consider a reconstruction procedure $\recon$, mapping a collection of cell averages $(\avgv_i)_{i\in\Z}$ to a piecewise polynomial function $\sum_i p_i\ind_{\cell_i}$. As before, define the cell interface values $v_\iphf^- = p_i(x_\iphf)$ and $v_\iphf^+ = p_{i+1}(x_\iphf)$ and the jump $\jmpr{v}_\iphf = v_\iphf^+-v_\iphf^-$. We say that $\recon$ satisfies the \emph{sign property} if for every $i\in\Z$,
\begin{equation}\label{eq:signpropexact}
\begin{split}
\text{if}\quad \jmp{\avgv}_\iphf > 0\quad & \text{then}\quad \jmpr{v}_\iphf \geq 0 \\
\text{if}\quad \jmp{\avgv}_\iphf < 0\quad & \text{then}\quad \jmpr{v}_\iphf \leq 0 \\
\text{if}\quad \jmp{\avgv}_\iphf = 0\quad & \text{then}\quad \jmpr{v}_\iphf = 0.
\end{split}
\end{equation}
As we have seen in Section \ref{sec:convhighorder}, the sign property implies that the diffusion coefficient in finite volume schemes of the form \eqref{eq:fvm}, \eqref{eq:entrstabflux} has the right sign. 

\begin{figure}
\centering
\subfigure[$k=2$]{\includegraphics[width=0.32\linewidth]{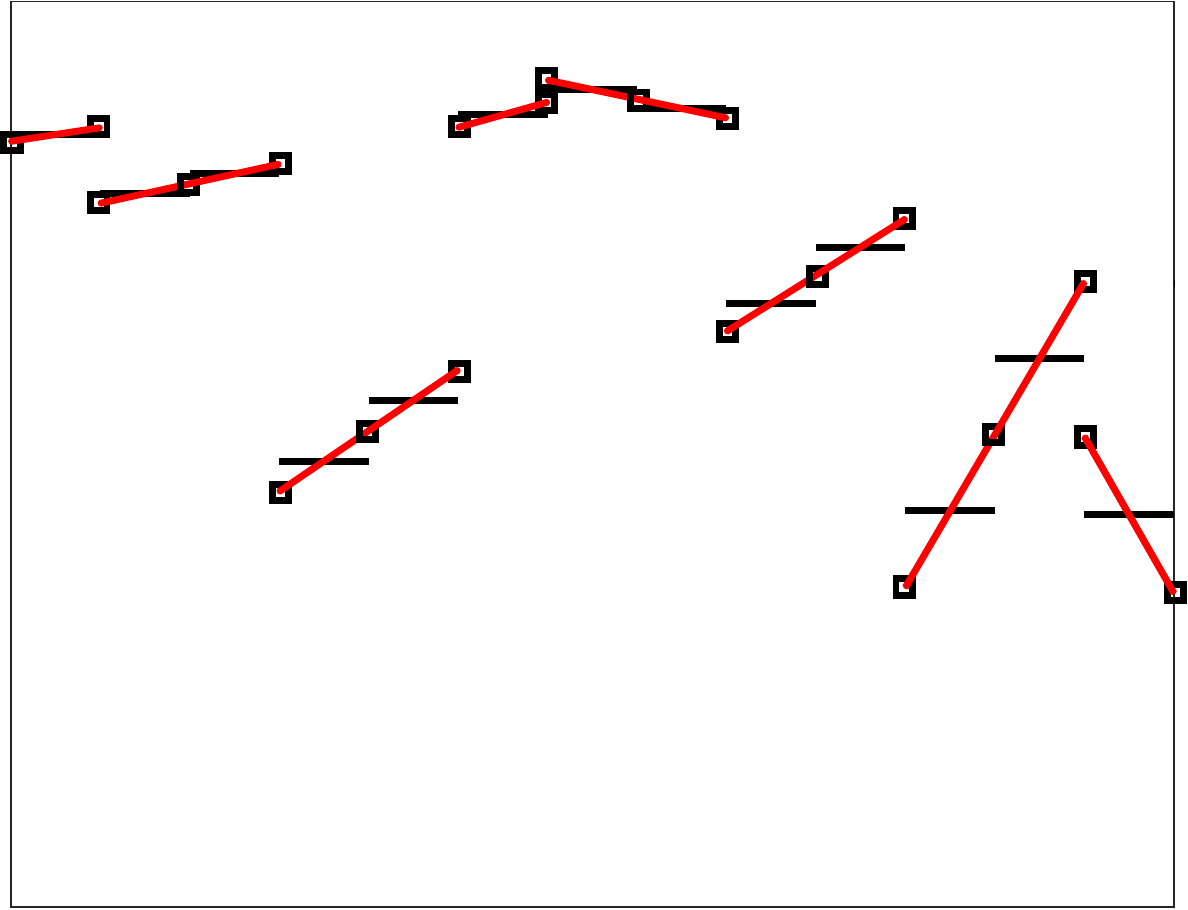}}
\subfigure[$k=3$]{\includegraphics[width=0.32\linewidth]{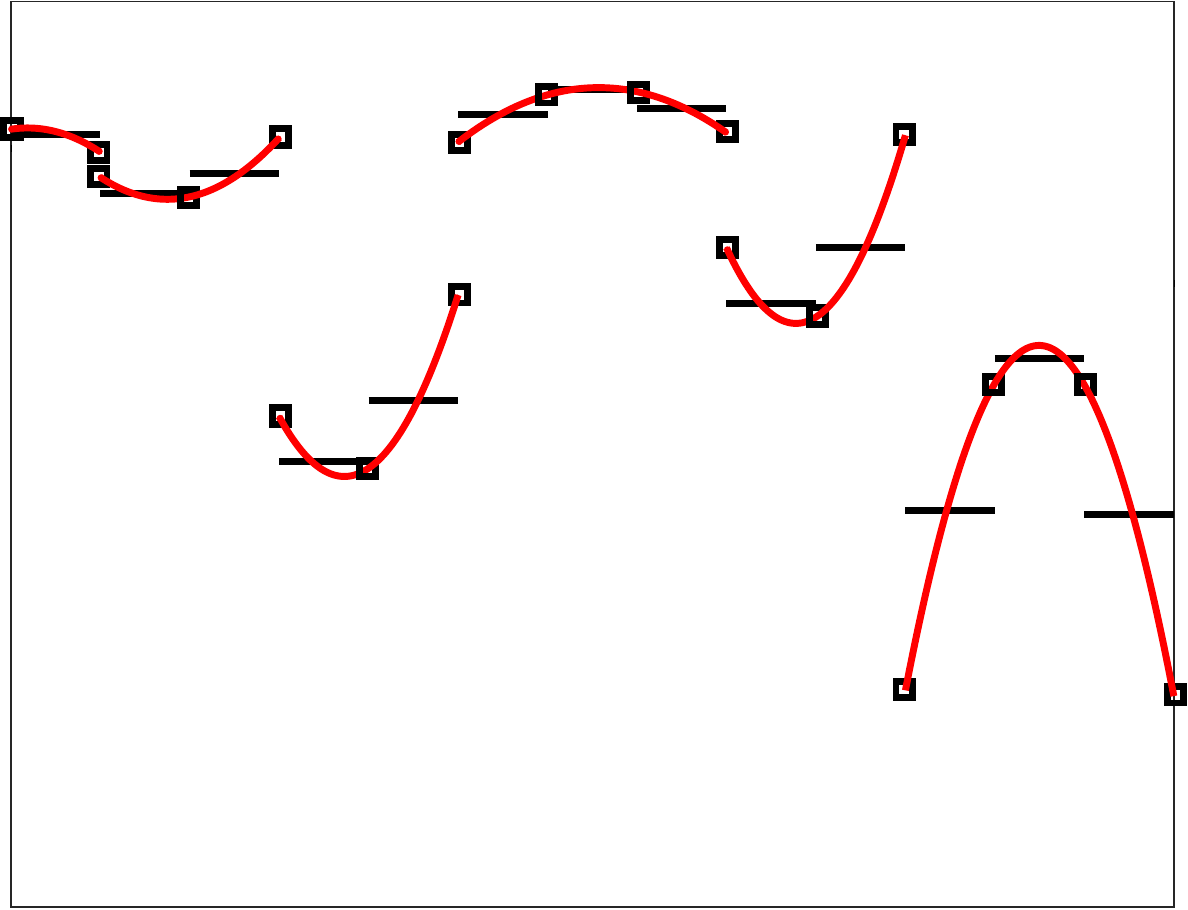}}
\subfigure[$k=4$]{\includegraphics[width=0.32\linewidth]{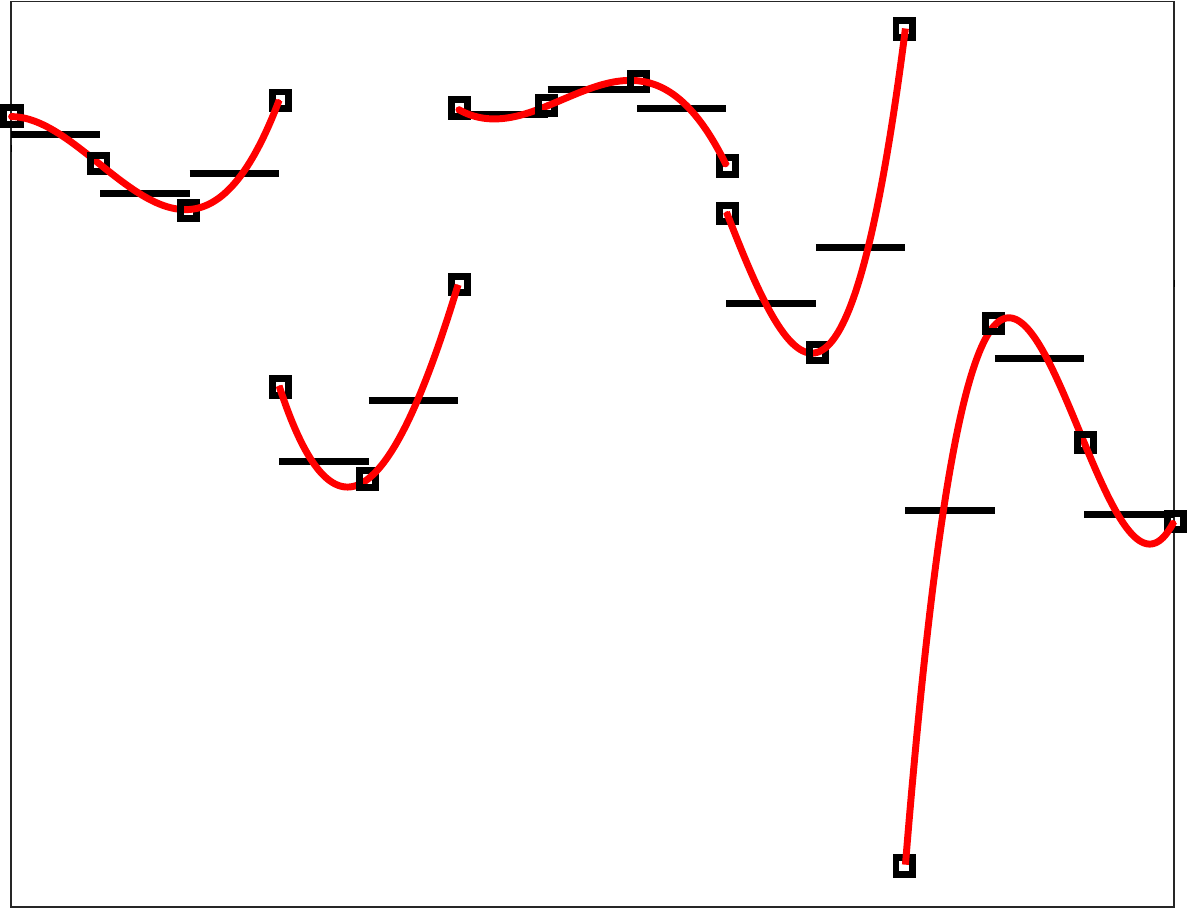}}
\caption{ENO reconstruction of randomly chosen cell averages. Black lines: cell averages. Red curves: reconstruction. Squares: cell interface values.}
\label{fig:eno}
\end{figure}
The sign property is illustrated in Figure \ref{fig:eno}, which shows a third-, fourth- and fifth-order ENO reconstruction of randomly chosen cell averages. Even though the reconstructed polynomial may have large variations within each cell, its jumps at cell interfaces always have the same sign as the jumps of the cell averages.

In \cite{Fjo13,FMT12a} it was shown that the $k$th order ENO reconstruction satisfies the sign property, for any $k\in\N$ and for any mesh $(x_\iphf)_{i\in\Z}$. We provide here a sketch of the proof.

\begin{proof}[Proof of ENO sign property (sketch)]
The first step is to derive the following expression for the jump in reconstructed values:
\begin{equation}\label{eq:jmpexpr}
\jmpr{v}_\iphf = \sum_{s=s_i}^{s_{i+1}-1} [\avgv_s,\dots,\avgv_{s+k}] X_{i,s}
\end{equation}
where
\[
X_{i,s} := (x_{s+k+\hf}-x_{s-\hf}) \prod_{\substack{m=0\\m\neq i-s}}^{k-1}(x_\iphf - x_{s+m+\hf}).
\]
When $s_i=s_{i+1}$, i.e.\ the neighboring stencils are the same, then \eqref{eq:jmpexpr} yields $\jmpr{v}_\iphf=0$ and the reconstruction is continuous across $x_\iphf$. Observe that \eqref{eq:jmpexpr} expresses $\jmpr{v}_\iphf$ in terms of only $k$th order divided differences of $\avgv$, instead of divided differences of order $1, \dots, k-1$, as one might expect from \eqref{eq:enoexpression}. In particular, when $k=1$ we get $\jmpr{v}_\iphf = [\avgv_i,\avgv_{i+1}](x_{i+\thf}-x_\imhf) = \jmp{v}_\iphf$, as expected.

The proof of \eqref{eq:jmpexpr} amounts to a simple manipulation of Newton polynomials, but the idea is quite clear: Both $v_\iphf^-$ and $v_\iphf^+$ are $k$th order approximations of $v(x_\iphf)$, with truncation terms of the form $[\avgv_s,\dots,\avgv_{s+k}] = \frac{1}{k!}\frac{d^k v}{dx^k}(\xi)$.

The next step is to show that each summand in \eqref{eq:jmpexpr} has the same sign as $\jmp{\avgv}_\iphf$. Because $\sgn( X_{i,s}) = (-1)^{s+k+1}$, we need only to show that
\begin{equation}\label{eq:signdivdiff}
\jmp{\avgv}_\iphf [\avgv_s,\dots,\avgv_{s+k}] (-1)^{s+k+1} \geq 0 \qquad \forall\ s=s_i^k,\dots,s_{i+1}^k-1.
\end{equation}
The proof of \eqref{eq:signdivdiff} is obvious for $k=1$. Assume that \eqref{eq:signdivdiff} holds for some $k\geq1$. It suffices to consider the case $\jmp{\avgv}_\iphf > 0$, so we have 
\[
[\avgv_s,\dots,\avgv_{s+k}] (-1)^{s+k+1} \geq 0 \qquad \text{for } s = s_i^k,\dots,s_{i+1}^k-1.
\]
The fact that $[\avgv_s,\dots,\avgv_{s+k+1}] (-1)^{s+k+2} \geq 0$ for $s = s_i^{k+1},\dots,s_{i+1}^{k+1}-1$ then follows by writing out the definition of these $(k+1)$th divided differences in terms of $k$th divided difference and using the induction hypothesis and the ENO choice of $s^{k+1}$. We refer to \cite{Fjo09,FMT12a} for the full proof.
\end{proof}

We emphasize that the sign property is mesh independent, in the sense that it holds for \textit{any} mesh $(x_\iphf)_{i\in\Z}$, regardless of the mesh width $\Dx_i$.

\begin{remark}
The ``point-value version'' of ENO (see Remark \ref{rem:pointeno}) also satisfies the sign property \eqref{eq:signpropexact}; see \cite{Fjo13,FMT12a}.
\end{remark}

\begin{remark}\label{rem:rdsignprop}
It is easy to confirm by numerical experiments that the ``RD'' (reconstruction with deconvolution) ENO method does \emph{not} satisfy the sign property. Indeed, Figure 3b of \cite{HEOC87}, which shows a fourth order RD ENO reconstruction, clearly violates the sign property at the fifth cell interface from the left.
\end{remark}

\subsection{Upper bound on jumps}\label{sec:upperjmpbound}
In \cite{Fjo13,FMT12a} it was shown that the ENO reconstruction procedure satisfies---in addition to the sign property---an upper bound on the jumps in the reconstructed polynomial. More precisely, for every $k\in\N$, the $k$th order ENO reconstruction satisfies
\begin{equation}\label{eq:upperbound}
0 \leq \frac{\jmpr{v}_\iphf}{\jmp{\avgv}_\iphf} \leq C_{k,i} \qquad \forall\ i\in\Z,
\end{equation}  
where $C_{k,i}$ depends only on $k$ and on the ratios $|\cell_j|/|\cell_\ell|$ of neighboring cell sizes. (Note that the first inequality in \eqref{eq:upperbound} is merely a restatement of the sign property \eqref{eq:signpropexact}.) Recall from Section \ref{sec:tecno} that his bound ensures Lipschitz continuity of the numerical flux \eqref{eq:entrstabflux}.

In the case of a uniform mesh, $|\cell_i|\equiv$ const., the constant $C_{k,i}\equiv C_k$ can be computed explicitly; see Table \ref{tab:upperBound}.
\begin{table}
\begin{center}
\begin{tabular}{|c|c|}
\hline $k$ & Upper bound $C_k$ \\
\hline 1 & 1 \\
\hline 2 & 2 \\
\hline 3 & $10/3 = 3.333\dots$ \\
\hline 4 & $16/3 = 5.333\dots$ \\
\hline 5 & $128/15 = 8.533\dots$ \\
\hline 6 & $208/15 = 13.866\dots$ \\
\hline
\end{tabular}
\end{center}
\caption{The upper bound in \eqref{eq:upperbound} for a uniform mesh.}
\label{tab:upperBound}
\end{table}
By way of an example, it was also found that the upper bound \eqref{eq:upperbound} is sharp. Indeed, if
\[
\avgv_i = \begin{cases} 
0 & \text{if $i$ is odd} \\
1 & \text{if $i$ is even and $i\leq 4$} \\
1-\eps & \text{if $i$ is even and $i > 4$.}
\end{cases}
\]
for any $\eps>0$, then the upper bound in \eqref{eq:upperbound} is attained in the limit $\eps\to0$. Figure \ref{fig:worstcase} shows these worst-case scenarios for $k=2,3,4,5$ and $\eps=10^{-10}$.
\begin{figure}
\centering
\includegraphics[width=0.42\linewidth]{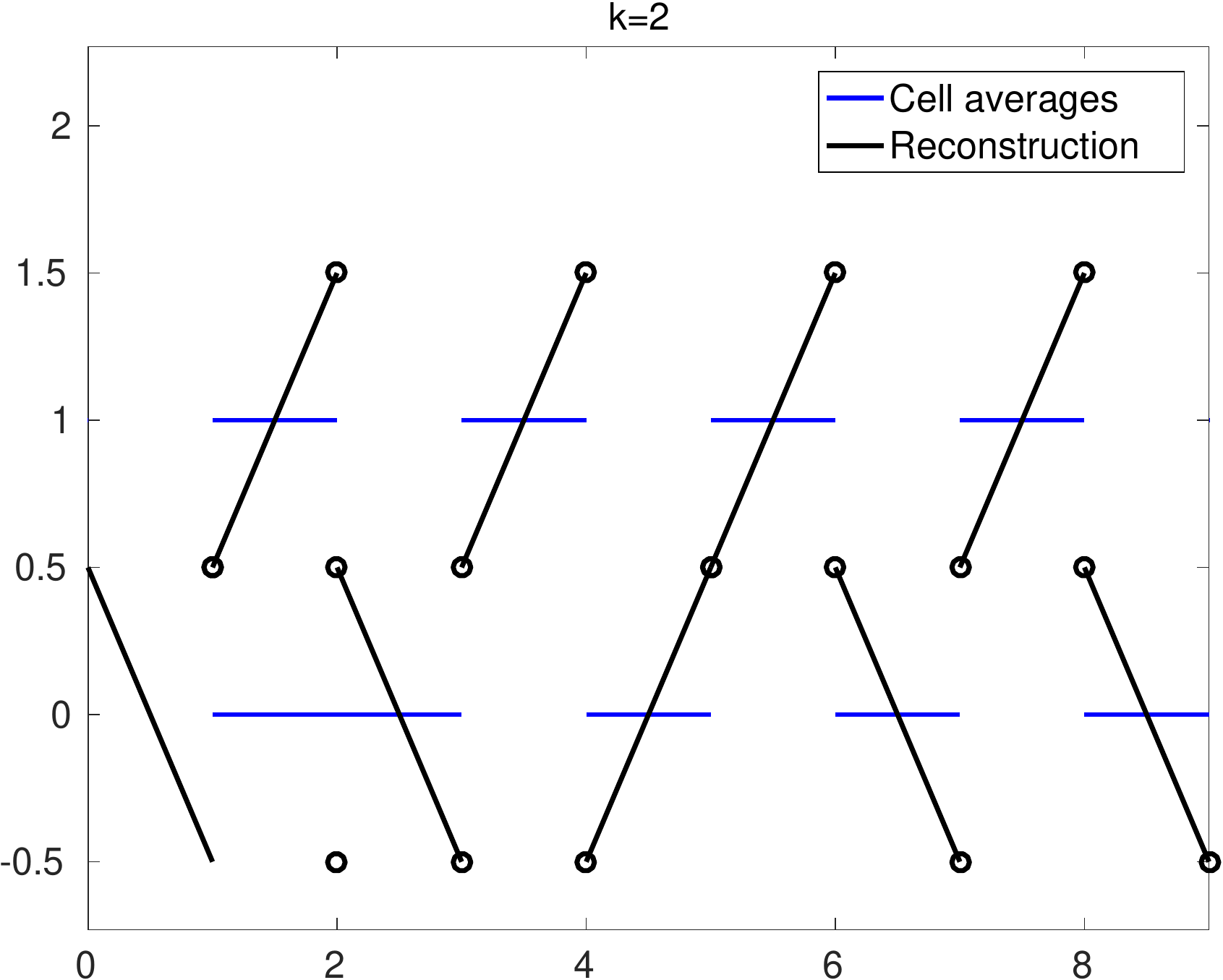}\hspace{2em}
\includegraphics[width=0.42\linewidth]{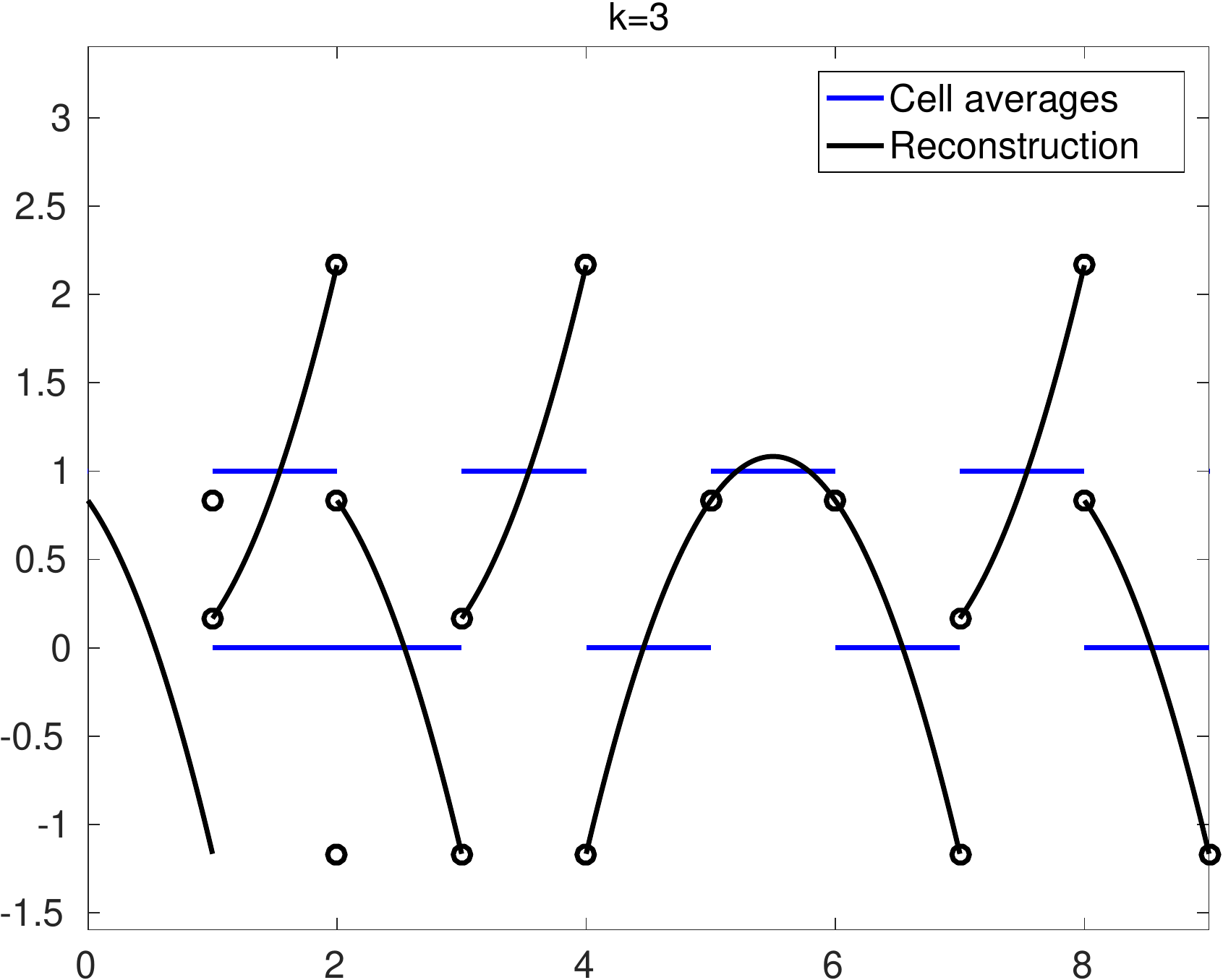} \\ \vspace{1em}
\includegraphics[width=0.42\linewidth]{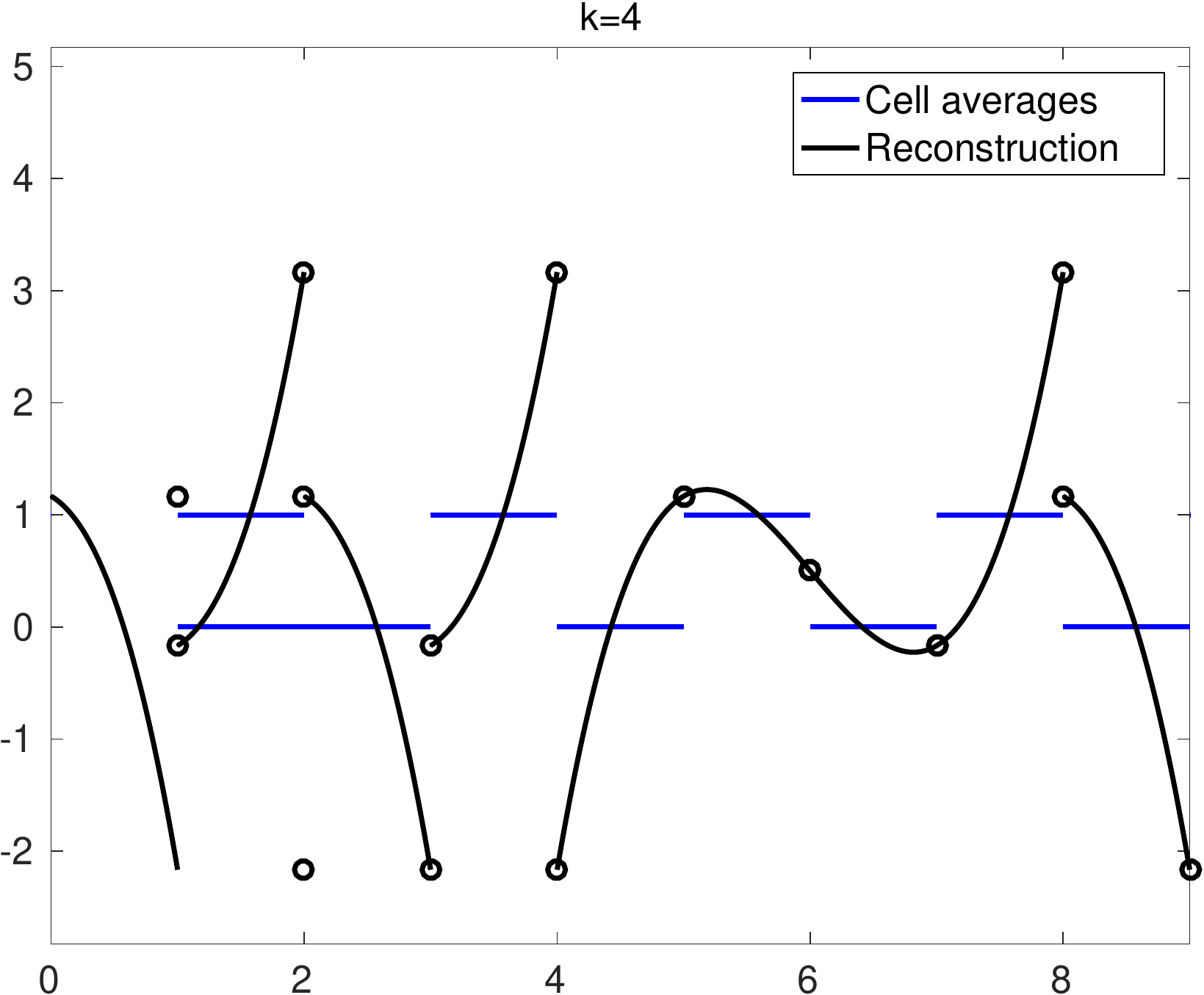}\hspace{2em}
\includegraphics[width=0.42\linewidth]{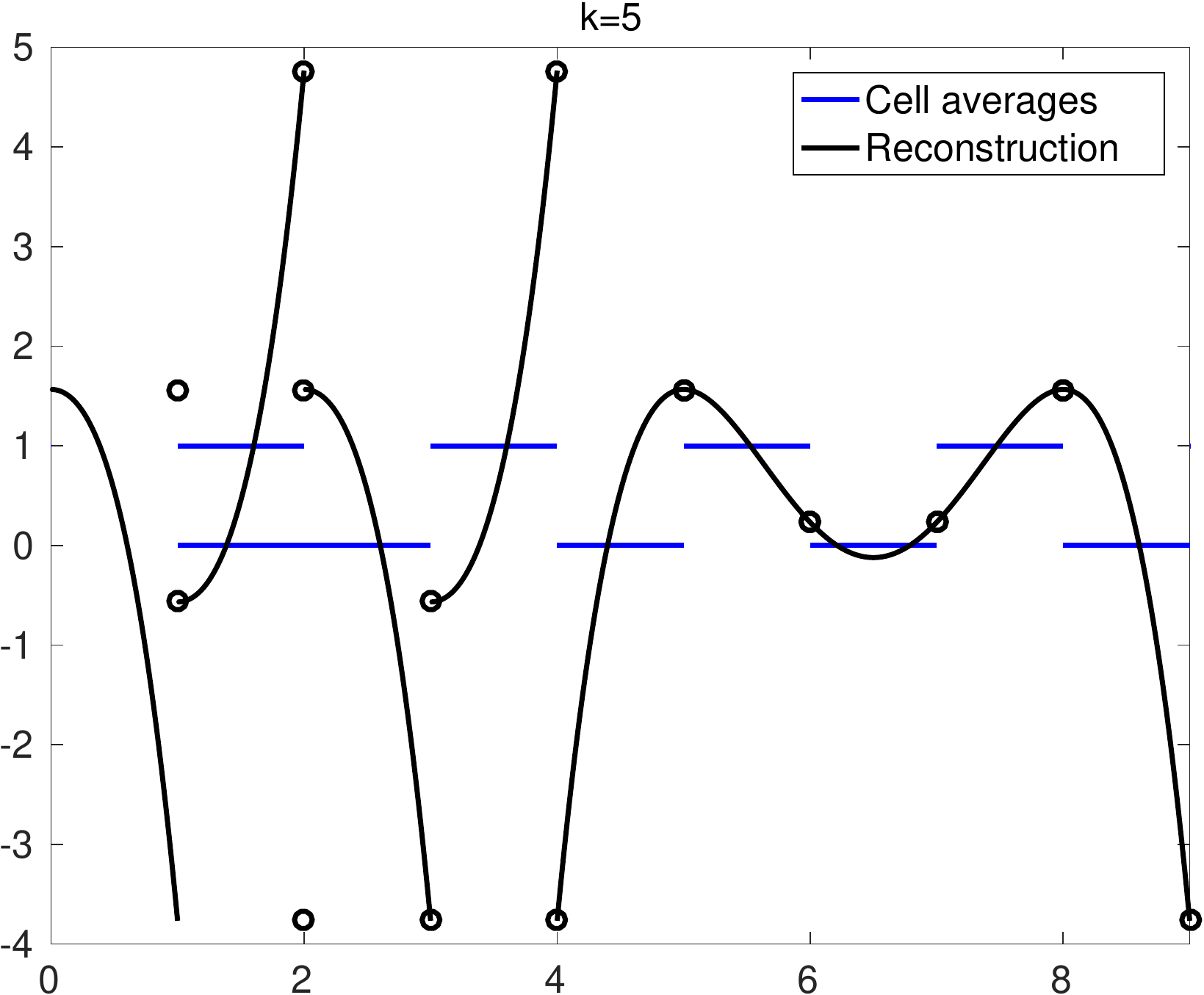}
\caption{Worst case cell interface jumps for $k=2, 3, 4, 5$.}
\label{fig:worstcase}
\end{figure}

\subsection{The ENO TV conjecture}\label{sec:enoconjecture}
Any compactness argument for numerical approximations of the conservation law \eqref{eq:cl} requires some ``weak TV bound'' of the form \eqref{eq:wtvbound}. To conclude such a bound on the basis of the available ``weak reconstructed TV bound'' \eqref{eq:wrecontvbound}, it would seem that a lower bound of the form $|\jmpr{v}_\iphf| \geq |\jmp{\avgv}_\iphf|$ for all $i$ is required. However, such a bound is impossible due to the possibility that $\jmpr{v}_\iphf=0$ even when $\jmp{\avgv}_\iphf \neq 0$ (see Section \ref{sec:disccelledges}).

In \cite{Fjo13} the following inequality was conjectured for the $k$th order ENO reconstruction method:
\begin{equation}\label{eq:tecnodiffusionbound}
\sum_{i\in\Z}|\jmp{\avgv}_\iphf|^{k+1} \leq C\|\avgv\|_{L^\infty}^{k-1}\sum_{i\in\Z}\jmp{\avgv}_\iphf\jmpr{v}_\iphf
\end{equation}
for some $C>0$ independent of $\avgv$ and $\Dx$. Clearly, if this were to hold then the ``weak reconstructed TV bound'' \eqref{eq:wrecontvbound}, together with an $L^\infty$ bound on $\avgv$, would imply \eqref{eq:wtvbound}. The only case for which \eqref{eq:tecnodiffusionbound} has been proven is for $k=2$, and we include the proof here. For the sake of simplicity we assume that the mesh is uniform.

\begin{proof}[Proof of \eqref{eq:tecnodiffusionbound} for $k=2$]
Denote $\Delta \avgv_i = \avgv_{i+1}-\avgv_i$, and iteratively $\Delta^k \avgv_i = \Delta^{k-1} (\Delta \avgv)_i$. The formula \eqref{eq:jmpexpr} yields
\begin{align*}
\sum_{i\in\Z}\jmp{\avgv}_\iphf\jmpr{v}_\iphf &= \sum_{i\in\Z}|\Delta \avgv_i| \sum_{j=s_i^2}^{s_{i+1}^2-1} a_i|\Delta^2 \avgv_j| \\*
&\geq a \sum_{i\in\Z}|\Delta \avgv_i| \sum_{j=s_i^2}^{s_{i+1}^2-1}|\Delta^2 \avgv_j|
\end{align*}
for constants $a_i \geq a > 0$ only dependent on $i$ and $s_i^2$. For every $j\in\Z$ there is precisely one index $i\in\Z$ such that $j\in \{s_i^2,\dots,s_{i+1}^2-1\}$, and we denote this index $i$ by $i=\iota_j^2$. Thus, we can write
\[
\sum_{i\in\Z}|\Delta \avgv_i| \sum_{j=s_i^2}^{s_{i+1}^2-1}|\Delta^2 \avgv_j| = \sum_{j\in\Z} |\Delta \avgv_{\iota_j^2}| |\Delta^2 \avgv_j|.
\]
It is straightforward to show that for $k=2$, the index $\iota$ is given by
\begin{equation}
\iota_j^2 = \begin{cases}
j & \text{if } |\Delta \avgv_j| > |\Delta \avgv_{j+1}| \\
j+1 & \text{if } |\Delta \avgv_j| \leq |\Delta \avgv_{j+1}|,
\end{cases}
\end{equation}
and as a consequence,
\begin{equation}\label{eq:iota2}
|\Delta \avgv_{\iota_j^2}| = \max\left(|\Delta \avgv_j|, |\Delta \avgv_{j+1}|\right).
\end{equation}
Starting with the left-hand side of \eqref{eq:tecnodiffusionbound} with $k=2$, we get
\begin{align*}
\sum_{i\in\Z} |\Delta \avgv_i|^3 &= \sum_{i\in\Z} |\Delta \avgv_i| \Delta \avgv_i \Delta \avgv_i \\
\text{\textit{(summation-by-parts)}}\qquad\qquad &= -\sum_{i\in\Z} \avgv_{i+1} \Delta\left(|\Delta \avgv_i|\Delta \avgv_i\right) \\
&= -\sum_{i\in\Z} \avgv_{i+1} \left(\left(\Delta|\Delta \avgv_i|\right)\Delta \avgv_i + |\Delta \avgv_i|\Delta^2 \avgv_i\right) \\
&\leq 2\sum_{i\in\Z} |\avgv_{i+1}||\Delta^2 \avgv_i||\Delta \avgv_i| \\
\text{\textit{(relabeling $i\mapsto j$ and using \eqref{eq:iota2})}}\qquad &\leq 2\|\avgv\|_{L^\infty} \sum_{j\in\Z} |\Delta^2 \avgv_j||\Delta \avgv_{\iota_j^2}|
\end{align*}
This completes the proof.
\end{proof}

\subsection{Mesh dependent properties}\label{sec:enomeshdep}
The ``mesh dependent properties'' of ENO are those properties which are satisfied asymptotically as $\Dx\to0$. In other words, for a \emph{fixed} underlying function $v(x)$, these are properties of ENO that are satisfied on \emph{sufficiently fine} meshes. Although these properties function as a proof-of-concept of the ENO reconstruction method, they are of limited value in applications to numerical methods for conservation laws \eqref{eq:cl} because for such applications, the cell averages in question will themselves depend (nonlinearly) on the mesh. As such, these properties cannot be used in a proof of stability or convergence of numerical schemes for \eqref{eq:cl}.

Below, we use the term ``shock'' to refer to any jump discontinuity of the underlying function $v$. For simplicity we will assume that $\Dx_i\equiv$ const.

\subsubsection{Uniform $k$th order accuracy up to discontinuities}\label{sec:uniformaccshock}
If $v$ is a {piecewise} $C^\infty$ function with finitely many jump discontinuities (``shocks''), then for \emph{sufficiently small $\Dx$}, the ENO reconstruction is a $k$th order approximation of $v$ in all cells not containing a shock \cite{HEOC86}. Indeed, if $\Dx$ is sufficiently small then there are at least $k$ cells in-between the shocked cells. Moreover, the $\ell$th divided difference $[\avgv_s,\dots,\avgv_{s+\ell}]$ over any stencil containing a shocked cell behaves as $O(\Dx^{-\ell})$. Thus, if $\Dx$ is small enough then in every non-shocked cell, the ENO stencil selection procedure can, and will, select an ENO stencil $\{s_i,\dots,s_i+k-1\}$ \textit{not} containing a shock. The property of uniform $k$th order accuracy then follows as in Section \ref{sec:uniformacc}.

\subsubsection{Monotonicity in shocked cells}\label{sec:monotonicity}
Harten et al.\ proved in \cite{HEOC86} that the primitive $P_i$ of the ENO reconstruction $p_i$ will be \emph{monotone} in every cell containing a discontinuity of $V$. This property is of limited value since \textit{(a)} the primitive $V$ is always continuous, and \textit{(b)} we are primarily interested in $p_i$, not $P_i$. However, it turns out that the same property in fact holds for the ENO reconstruction $p_i$ (see Figure \ref{fig:enoMonotone}).

\begin{figure}
\centering
\includegraphics[width=0.7\textwidth]{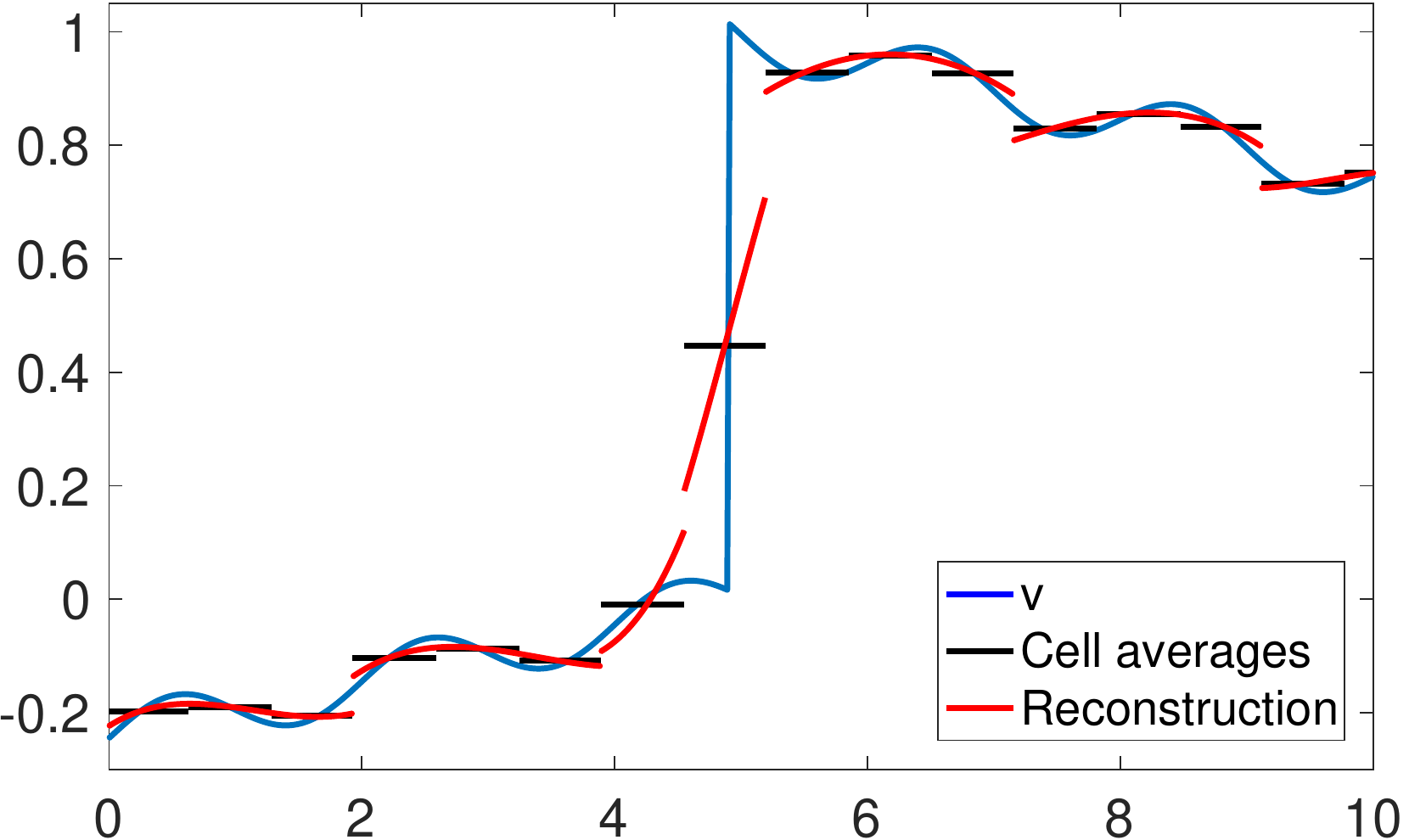}	
\caption{Monotonicity of fourth-order ENO reconstruction in a shocked cell.}
\label{fig:enoMonotone}
\end{figure}
\begin{proposition}\label{prop:monotonicity}
Let $v$ be a piecewise $C^\infty$ function with finitely many shocks. Then for sufficiently small $\Dx$, the ENO reconstruction of $v$ is monotone in every cell containing a jump discontinuity---more precisely, it is strictly increasing at positive jumps and strictly decreasing at negative jumps.
\end{proposition}

\begin{proof}
\newcommand{\is}{i}
The proof is similar in spirit to the proof of \cite[Theorem 4.1]{HEOC86}.

By choosing $\Dx$ sufficiently small, we may assume that shocked cells are at least $k$ cells from one another, and hence it suffices to consider the case $v=w+H$, where $w$ is Lipschitz continuous and $H$ is piecewise constant with a single jump discontinuity at $x=\bar{x}\in(x_{\is-\hf},x_{\is+\hf})$, for some index $\is\in\Z$. By the linearity of the ENO method (see Section \ref{sec:enolinear}), we may assume that 
\[
H(x) = \begin{cases}
0 & \text{if } x<\bar{x}\\ 
1 & \text{if } x>\bar{x}.
\end{cases}
\]
Moreover, we may assume that $k\geq3$ since the cases $k=1$ (piecewise constant reconstruction) and $k=2$ (piecewise linear reconstruction) are immediate.

Let $S = \{s_\is^k, \dots, s_\is^k+k-1\}$ denote the ENO reconstruction stencil in cell $\is$ and let $\cell=\bigcup_{j\in S} \cell_j$. We can write $p_\is = q + G$, where $q$ and $G$ are $(k-1)$th order polynomials which interpolate $(\avg{w}_j)_{j\in S}$ and $(\avg{H}_j)_{j\in S}$, respectively. Since $w$ is Lipschitz continuous we have $|[\avg{w}_j,\dots,\avg{w}_{j+\ell}]|\leq C \Dx^{-\ell}$ for all $j,\ell$, so from \eqref{eq:enoexpression} we get
\begin{equation}\label{eq:qlipschitz}
\left\|\frac{dq}{dx}\right\|_{L^\infty(\cell)} \leq C
\end{equation}
for some $C$ independent of $\Dx$.

Since $G$ interpolates $(\avg{H}_j)_{j\in S}$, there is at least one point $y_j \in \cell_j$ for every $j\in S$, $j\neq \is$ such that 
\[
G(y_j) = H(y_j) = \begin{cases}
0 & \text{if } j<\is \\
1 & \text{if } j>\is.
\end{cases}
\]
If there is more than one such root in cell $\cell_j$ we select the root $y_j$ which is closest to $\bar{x}$. By Rolle's theorem, the function $\frac{dG}{dx}$ has a zero in every interval of the form
\begin{equation}\label{eq:rolleintervals}
\begin{split}
(y_{j-1},y_j) \qquad &\text{for } s_\is < j < \is \\
(y_j, y_{j+1}) \qquad &\text{for } \is < j < s_\is+k-1.
\end{split}
\end{equation}
Note that cell $\cell_\is$ intersects none of the above intervals. We will show that $\frac{dG}{dx}$ cannot have a zero in $(y_{\is-1},y_{\is+1})\supset \cell_\is$. Choosing $\Dx$ small enough and using \eqref{eq:qlipschitz}, we can then conclude that also $p_\is=q+G$ must be monotone in $\cell_\is$.

We divide into two cases:

\noindent
\textbf{Case 1:} $s_\is \in \{\is, \is-k+1\}$, i.e.\ there are no cells in the stencil either to the left or to the right of $\cell_\is$. In this case there are exactly $k-2$ intervals of the form \eqref{eq:rolleintervals}. Since the $(k-2)$th order polynomial $\frac{dG}{dx}$ can have at most $k-2$ zeros, it cannot have another zero in $\cell_\is$.

\noindent
\textbf{Case 2:} $s_\is \notin \{\is, \is-k+1\}$. In this case there are exactly $k-3$ intervals of the form \eqref{eq:rolleintervals}. From the jump expression \eqref{eq:jmpexpr} and the sign property (see Section \ref{sec:signprop}), we get
\begin{align*}
(p_{\is+1}-p_\is)(x_{\is+\hf}) &= \sum_{s=s_\is}^{s_{\is+1}-1}[\avgv_s,\dots,\avgv_{s+k}]X_{\is,s} \\*
&\geq \sum_{s=s_\is}^{s_{\is+1}-1}|[\avgv_s,\dots,\avgv_{s+k}]|\Dx^k \\*
&\geq b_{\is+\hf}
\end{align*}
for some $b_{\is+\hf} > 0$ independent of $\Dx$. (Here, we have used the fact that $[\avgv_s,\dots,\avgv_{s+k}] \sim \Dx^{-k}$ for all $s\in\{\is-k,\dots,\is\}$). Similarly, we get
\[
(p_\is-p_{\is-1})(x_{\is-\hf}) \geq b_{\is-\hf}
\]
for some $b_{\is-\hf} > 0$ independent of $\Dx$. Thus,
\[
\begin{cases}
G(x_{\is-\hf}) = (p_\is-q)(x_{\is-\hf}) \geq b_{\is-\hf} + (p_{\is-1}-q)(x_{\is-\hf}) = b_{\is-\hf} + O(\Dx), \\
G(x_{\is+\hf}) = (p_\is-q)(x_{\is+\hf}) \leq -b_{\is+\hf} + (p_{\is+1}-q)(x_{\is+\hf}) = 1-b_{\is+\hf} + O(\Dx).
\end{cases}
\]
Choosing $\Dx$ small enough that the ``$O(\Dx)$'' terms are smaller than $b_{\is\pm\hf}$, we find that
\[
G(y_{\is-1}) = 0, \qquad G(x_{\is-\hf})>0, \qquad G(x_{\is+\hf})<1, \qquad G(y_{\is+1})=1,
\]
and hence,
\[
\frac{dG}{dx}(y_{\is-1}) \geq 0, \qquad \frac{dG}{dx}(y_{\is+1}) \geq 0.
\]
Thus, if $\frac{dG}{dx}$ has a zero in $(y_{\is-1},y_{\is+1})$, there must be at least two of them (or one zero with multiplicity at least 2). But the $(k-2)$th order polynomial $\frac{dG}{dx}$ already has $k-3$ zeros in the intervals \eqref{eq:rolleintervals}, so it cannot any zeros in $(y_{\is-1},y_{\is+1}).$
\end{proof}

\subsubsection{Essentially non-oscillatory}
The ``essentially non-oscillatory'' property, from which ENO derives its name, can be roughly stated as follows: Up to a term of order $\Dx^{k}$, the total variation of the ENO reconstruction $p$ is less than that of $v$. As with the monotonicity property, Harten et al.\ \cite{HEOC86} proved this only for the primitives $P$, $V$, not for the reconstruction $p$ itself. However, with Proposition \ref{prop:monotonicity} in place we can establish this result also for $p$.

\begin{theorem}
Assume that $v$ is piecewise $C^\infty$ with finitely many jump discontinuities. Then for sufficiently small $\Dx$, there exists a function $z=z(x)$ such that 
\[
z(x) = p(x) + O(\Dx^{k})\ \forall\ x, \qquad \TV(z) \leq \TV(v),
\]
where $p=\recon(\avgv)$ is the ENO reconstruction of $v$.
\end{theorem}
\begin{proof}
Let $\Dx$ be sufficiently small that $p(x) = v(x) + O(\Dx^k)$ in all non-shocked cells (see Section \ref{sec:uniformaccshock}). Decrease $\Dx$ further such that $p$ is monotone in all shocked cells (see Section \ref{sec:monotonicity}). We choose $z(x)=v(x)$ in non-shocked cells, and $z(x)=p(x)$ in shocked cells. After an $O(\Dx^{k})$ modification near the interfaces $x_\iphf$ between shocked and non-shocked cells, the sign property implies that the total variation does not increase at these points.
\end{proof}

\begin{remark}
Although the above theorem says nothing about $\TV(p)$, it may be shown that $\TV(p) \leq \TV(v) + O(\Dx^k)$ for sufficiently small $\Dx$.
\end{remark}

\subsection{ENO deficiencies}\label{sec:enodeficiency}
Despite satisfying numerous stability properties, the ENO reconstruction method suffers from some deficiencies which makes it less attractive for certain applications such as numerical methods for linear conservation laws.

\subsubsection{$\recon$ is discontinuous}
The ENO reconstruction $\recon : (\avgv_i)_{i\in\Z} \to p$ is discontinuous, in the sense that a small change in $\avgv_i$ (such as round-off errors) can change the switch in the ENO stencil selection procedure, thus producing a different reconstruction $p_j$. Although this stencil switching might not be a problem in practice, the discontinuous nature of ENO-based methods makes their analysis significantly more difficult.

\subsubsection{Inefficient use of information}
Although the final ENO reconstruction $p_i$ in a cell only relies on $k$ values, the ENO stencil selection procedure depends on all $2k-1$ neighboring points. This is an inefficient use of information; using all $2k-1$ points would potentially give up to $(2k-1)$th order accuracy in smooth parts of the solution. This situation is exacerbated in multiple dimensions. 

The WENO method uses a much more compact interpolation stencil and might therefore be more suitable for multi-dimensional problems.

\subsubsection{Instabilities in linear problems}
\begin{figure}
\centering
\includegraphics[width=0.31\textwidth]{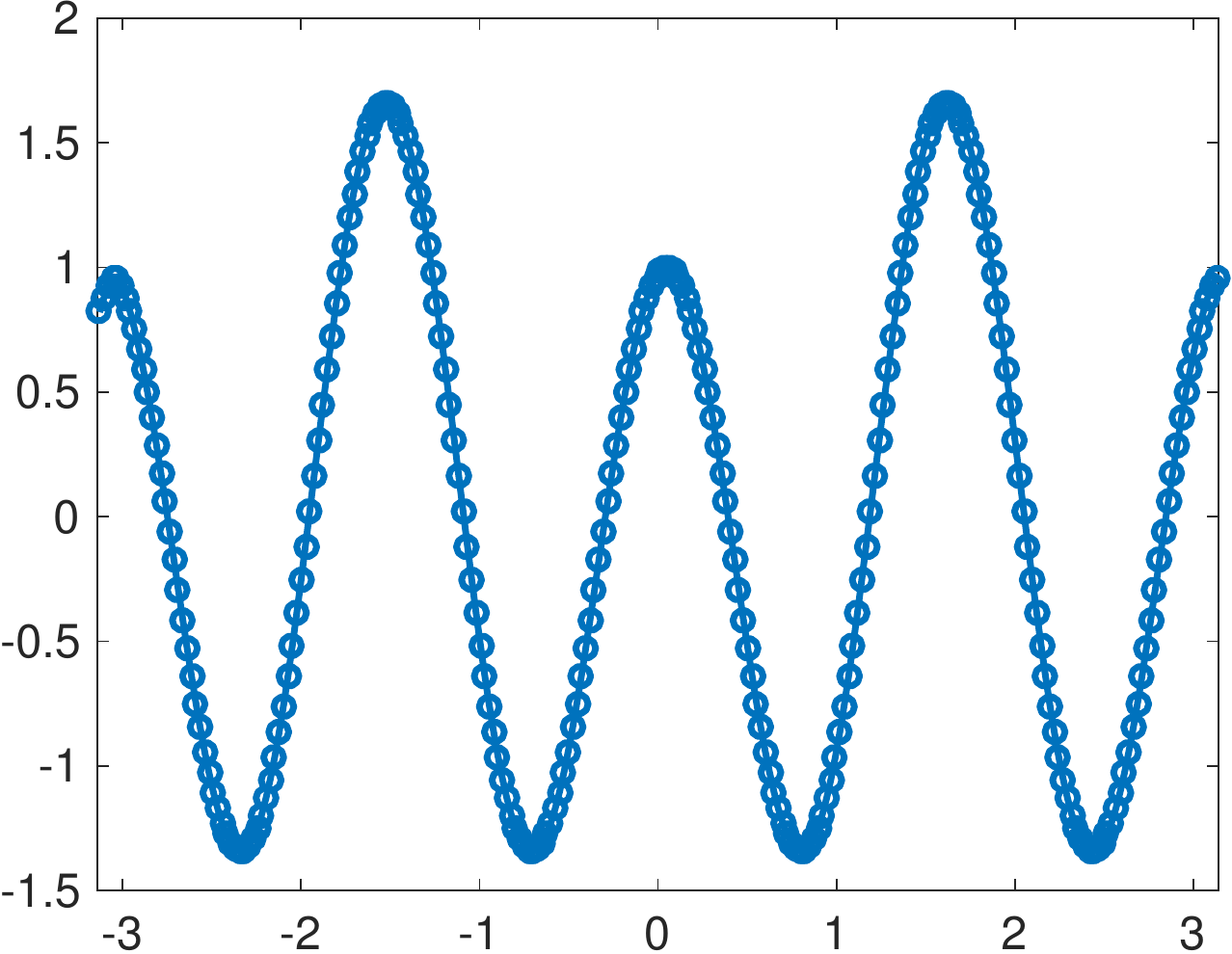} \hspace{1ex}
\includegraphics[width=0.31\textwidth]{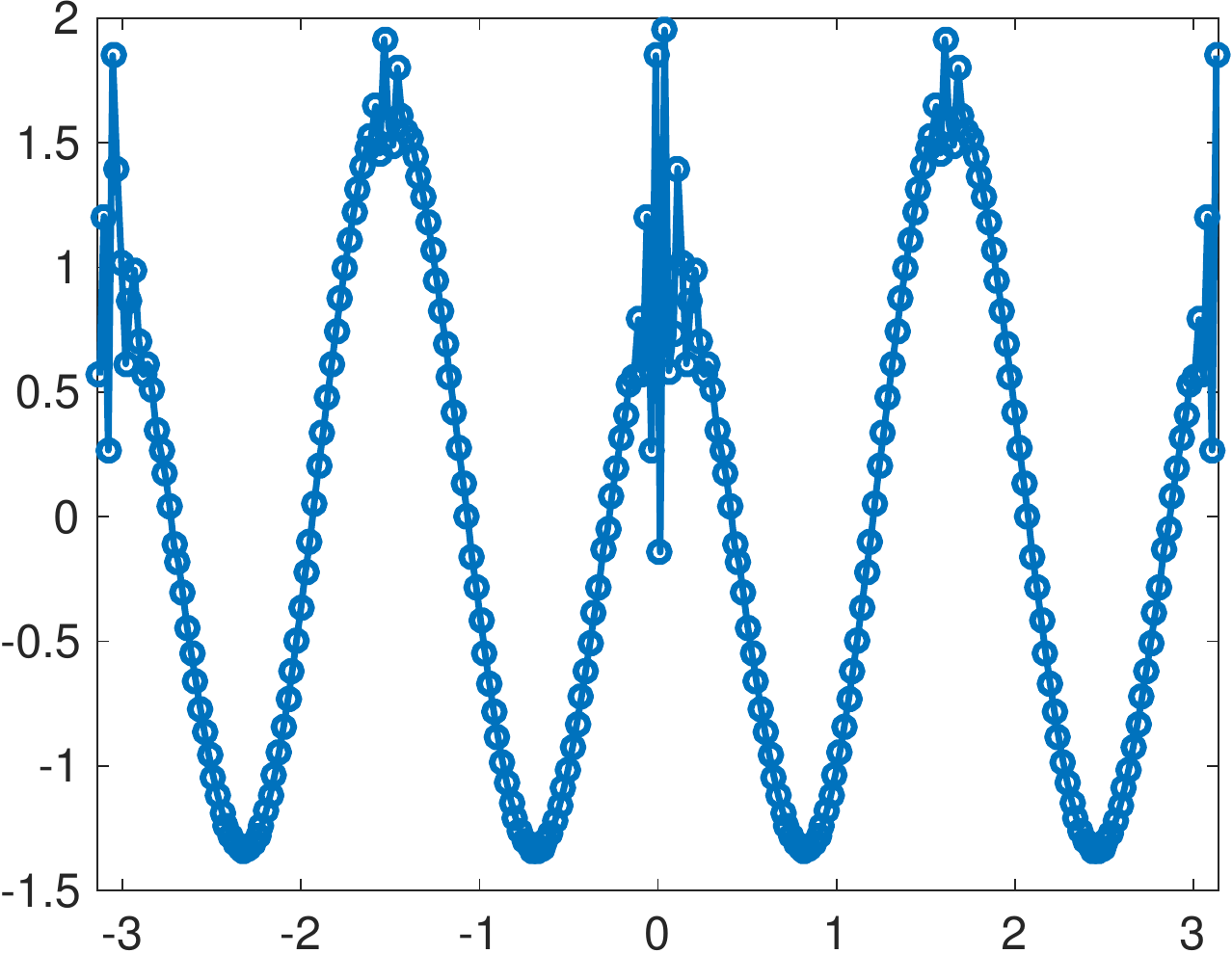} \hspace{1ex}
\includegraphics[width=0.30\textwidth]{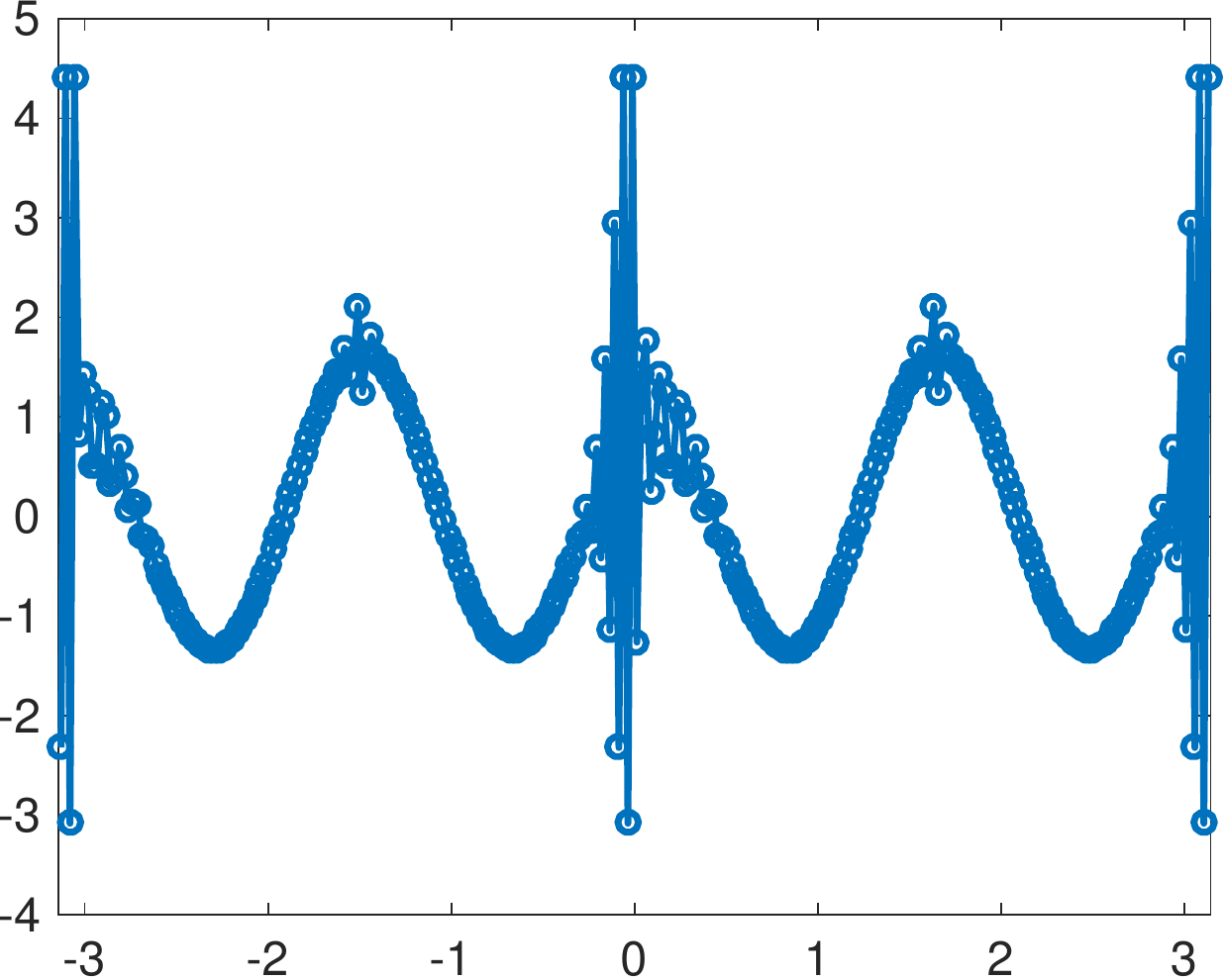} \\ \vspace{1em}
\includegraphics[width=0.3\textwidth]{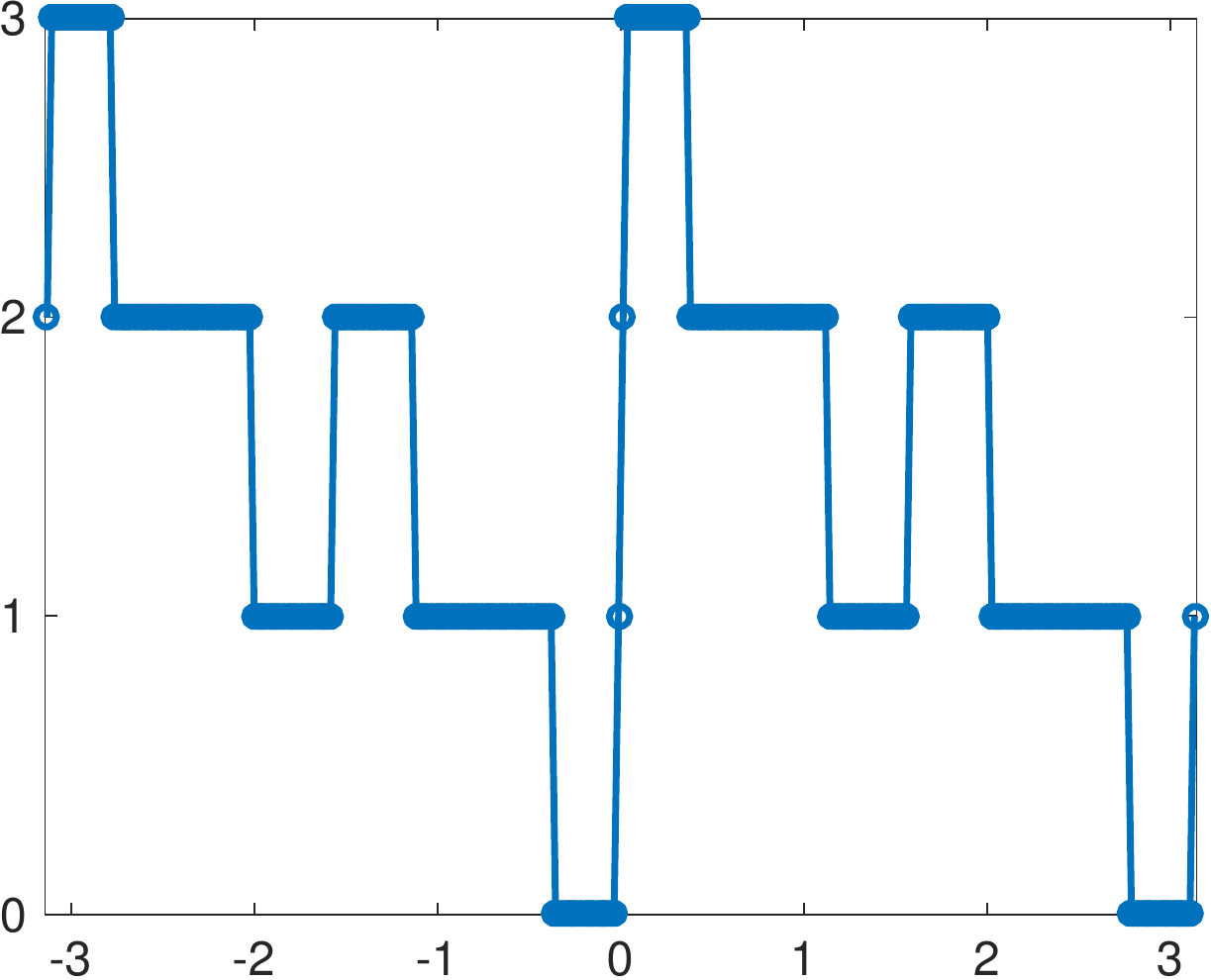} \hspace{1ex}
\includegraphics[width=0.3\textwidth]{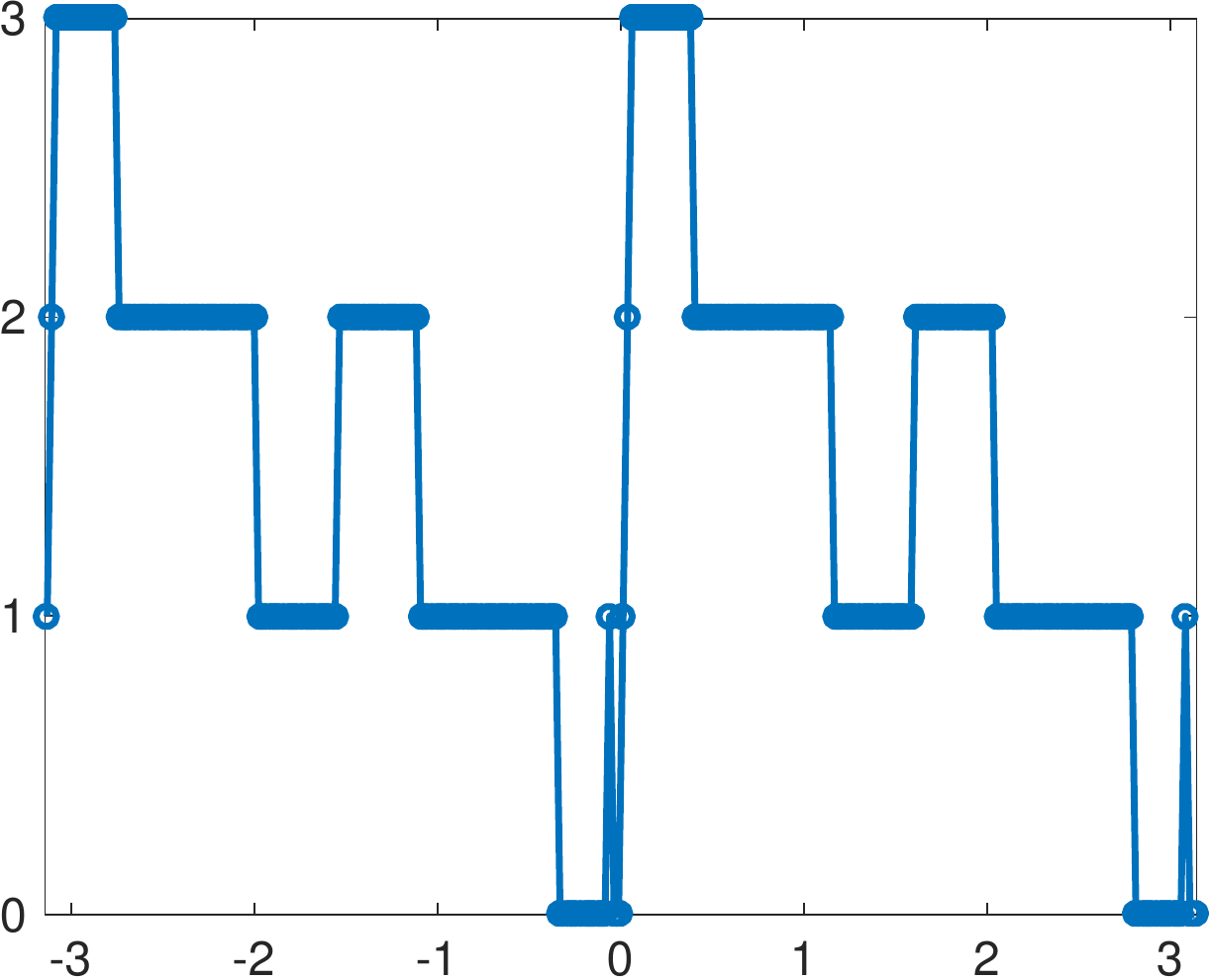} \hspace{1ex}
\includegraphics[width=0.3\textwidth]{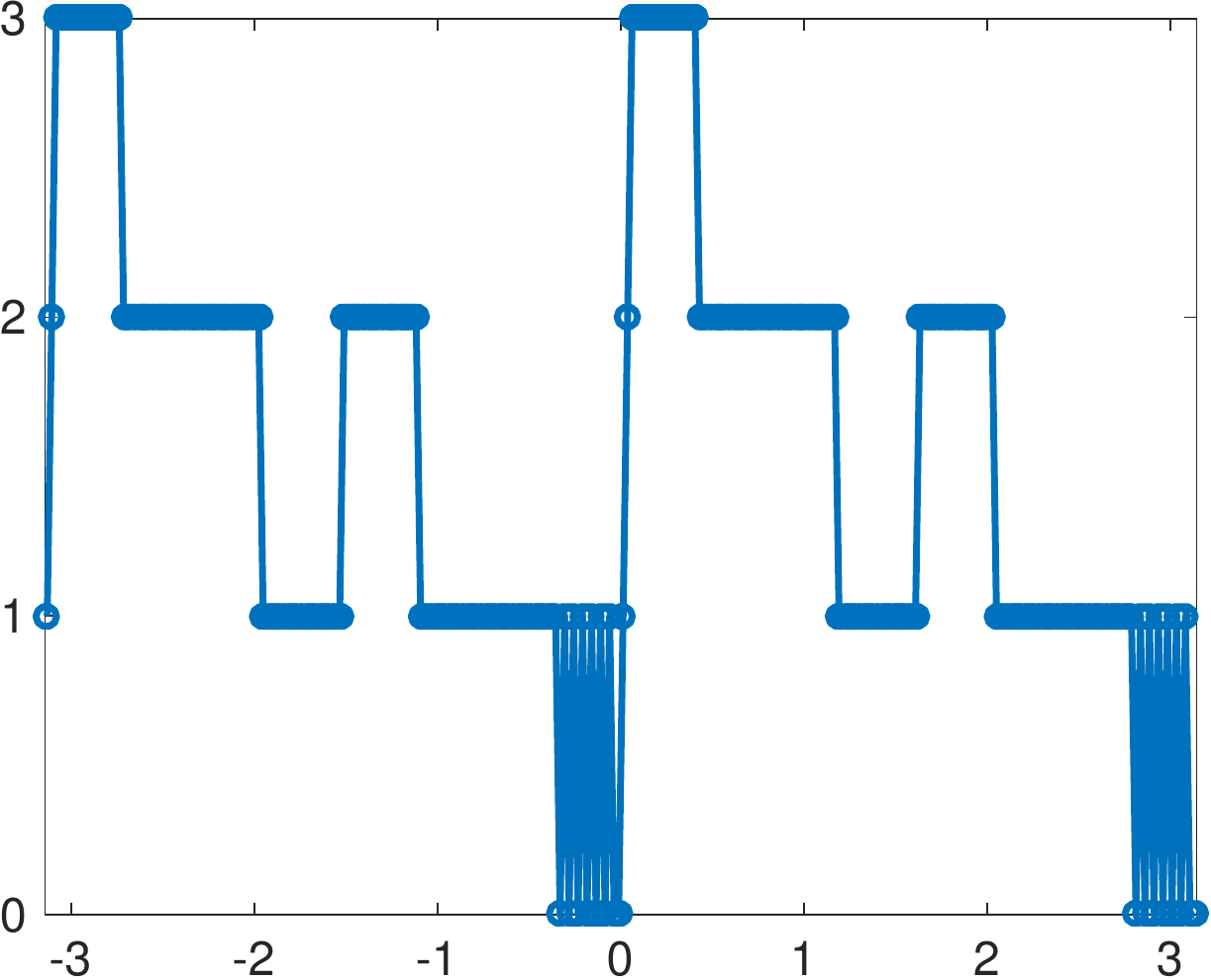}
\caption{Fourth-order divided differences (top row) and stencil offset $r_i$ (bottom row) at $t=0$ (left), $t\approx 0.02$ (middle) and $t\approx 0.04$ (right).}
\label{fig:sin4}
\end{figure}
Rogerson and Meiburg \cite{RM90} reported on a series of numerical experiments with an ENO-based fourth-order finite difference schemes for the periodic linear advection equation
\[
\partial_t u + \partial_x u = 0, \qquad x\in[-\pi,\pi).
\]
They observed the expected fourth-order convergence rate with $u_0(x) = \sin(x)$, but with $u_0(x)=\sin(x)^4$ they observed a decay in the convergence rate at moderately high values of $N$, the number of meshpoints. 

We approximate the above initial value problem using a Godunov-type finite volume scheme with fourth-order ENO reconstruction and a fourth-order Runge-Kutta time integrator. (Rogerson and Meiburg computed with the so-called ENO-Roe method \cite{SO89}, but the problem persists in other variants of ENO method and hence seems to be inherent to the ENO reconstruction procedure.) Figure \ref{fig:sin4} (top row) shows the fourth-order divided difference $[\avgv_i,\dots,\avgv_{i+4}]$ at various times. High-frequency oscillations appear quickly at the critical points $x=0$, $x=\pm\frac{\pi}{2}$, $x=\pm\pi$, and over time these oscillations propagate into the lower-order divided differences, finally polluting the solution $\avgv_i$. The oscillations near $x=\pm\frac{\pi}{2}$ stay bounded, whereas the oscillations near $x=0$, $x=\pm\pi$ grow unboundedly. 

Figure \ref{fig:sin4} (bottom row) shows the \textit{stencil offset} $r_i^k = i -s_i^k \in \{0,\dots,k-1\}$. (Recall from Section \ref{sec:disccelledges} that every interface $x_\iphf$ where $r_{i+1}\geq r_i$ will have a discontinuity in the reconstruction, which might lead to larger truncation errors.) Near the oscillatory points $x=0$, $x=\pm\pi$, the ENO method selects the stencils  $r_i^3=0$ and $r_i^3 = 3$. Rogerson and Meiburg \cite{RM90} call these stencils \textit{linearly unstable}: setting  $r_i^3 \equiv 0$ or $\equiv3$ for all $i$ will give an unconditionally unstable, divergent scheme, whereas $r_i\equiv1$ or $\equiv2$ gives a stable, convergent scheme. Although this heuristic explanation might very well be the root of the problem, the nonlinear nature of ENO makes this problem very hard to analyze rigorously. Further discussion can be found in \cite[Section 5]{AL93}, \cite{Har87b} and \cite{Shu90}. We mention in closing that the WENO method does not exhibit these instabilities for this particular problem \cite{Shu97}.

\section{Summary}
The ENO method has been enormously influential in the numerics community for hyperbolic conservation laws. Despite of its highly nonlinear (even discontinuous) nature, it yields expressions and formulas which are rather easy to analyze, and enjoys several surprising properties such as the non-oscillatory property, the sign property and upper bounds on discontinuities. As discussed in Section \ref{sec:enodeficiency}, certain ENO-based finite volume methods suffer from instabilities which prevent convergence. A rigorous analysis of this problem would be highly interesting (not to mention difficult), and might lead to provably stable ENO-type methods.


\end{document}